\documentclass[a4paper,12pt]{amsart}

\usepackage[T1]{fontenc}

\usepackage{lmodern}

\usepackage{microtype}

\usepackage{hyperref}

\usepackage{amsmath,amssymb,enumerate,color,comment}

\usepackage[abbrev,backrefs]{amsrefs}

\DeclareMathOperator{\Ad}{Ad}
\DeclareMathOperator{\ad}{ad}
\newcommand{\g}{\mathfrak{g}}
\renewcommand{\k}{\mathfrak{k}}
\newcommand{\p}{\mathfrak{p}}
\renewcommand{\a}{\mathfrak{a}}
\newcommand{\n}{\mathfrak{n}}
\newcommand{\s}{\mathfrak{s}}
\newcommand{\M}{\mathbb{M}}
\renewcommand{\div}{\operatorname{div}}

\newtheorem{thm}{Theorem}
\newtheorem{prop}[thm]{Proposition}
\newtheorem{lem}[thm]{Lemma}

\begin{document}

\title[Bourgain-Brezis inequalities on symmetric spaces]{Bourgain-Brezis inequalities on symmetric spaces of non-compact type}
\author[S. Chanillo]{Sagun Chanillo}
\address{Department of Mathematics\\ Rutgers, the State University of New
Jersey\\ 110 Frelinghuysen Road\\ Piscataway, NJ 08854, USA}
\email{chanillo@math.rutgers.edu}
\author[J. Van Schaftingen]{Jean Van Schaftingen}
\address{Institut de Recherche en Math\'ematique et en Physique\\ Universit\'e catholique de Louvain\\ Chemin du Cyclotron 2 bte L7.01.01\\ 1348 Louvain-la-Neuve\\ Belgium}
\email{Jean.VanSchaftingen@uclouvain.be}
\author[P.-L. Yung]{Po-Lam Yung}
\address{Department of Mathematics\\ the Chinese University of Hong Kong\\
Shatin\\ Hong Kong}
\email{plyung@math.cuhk.edu.hk}

\keywords{Critical Sobolev space; symmetric space; Lie group; Iwasawa decomposition.}
\subjclass[2010]{
35R03 
 (%
43A80, 
53C35
)}
\maketitle

\begin{abstract}
Let $\M$ be a global Riemannian symmetric space of non-compact type.
We prove a duality estimate, for pairings of divergence-free $L^1$ vector fields, with vector fields in a critical Sobolev space on $\M$:
\[
  \left| \int_{\M} \langle f, \phi \rangle dV \right| \leq C \Vert f\Vert_{L^1(dV)} \Vert \nabla \phi\Vert_{L^m(dV)}.
\]
This estimate provides a remedy for the failure of a critical Sobolev embedding on such symmetric spaces.
\end{abstract}

\section{Introduction}

When \(n \ge 2\), it is known that there is no embedding of the homogeneous Sobolev space
\(\dot{W}^{1, n } (\mathbb{R}^n)\) in \(L^\infty (\mathbb{R}^n)\).
However, Bourgain and Brezis \citelist{\cite{MR1949165}\cite{MR1913720}\cite{MR2293957}\cite{MR2057026}}
have observed that if \(f\) is a smooth divergence-free vector field on $\mathbb{R}^n$, and \(\phi\) is a smooth compactly supported vector field on $\mathbb{R}^n$, then the following estimate holds:
\begin{equation}
\label{ineqBBVS}
  \left| \int_{\mathbb{R}^n} \langle f, \phi \rangle \right| \leq C \Vert f\Vert_{L^1} \Vert \nabla \phi\Vert_{L^n}.
\end{equation}
This can be thought of as some compensation phenomenon arising from the divergence-free condition on \(f\), and allows in some sense to remedy the failure of the embedding of $\dot{W}^{1,n}(\mathbb{R}^n)$ into $L^{\infty}(\mathbb{R}^n)$. The estimate \eqref{ineqBBVS} is slightly stronger than the embedding of \(\dot{W}^{1, n}(\mathbb{R}^n)\) in BMO \cite{MR2240172}.

Bourgain and Brezis have deduced the estimate \eqref{ineqBBVS} from the solvability of the Hodge-de Rham system: they showed that for every differential $\ell$-form \(\eta\) on $\mathbb{R}^n$ with coefficients in the critical homogeneous Sobolev space $\dot{W}^{1,n}(\mathbb{R}^n)$, where $1 \leq \ell \leq n-1$, there exists a differential $\ell$-form $\theta$, whose components are all in $\dot{W}^{1,n} \cap L^{\infty} (\mathbb{R}^n)$, such that
\(d \theta = d \eta,\)
with
\[
\Vert \theta \Vert_{\dot{W}^{1,n}} + \Vert  \theta \Vert_{L^{\infty}} \leq C \Vert d \eta\Vert_{L^n},
\]
for some constant \(C\) that does not depend on \(\eta\).
The construction is based on a Littlewood--Paley decomposition. Their result implies in fact a stronger estimate than \eqref{ineqBBVS}, in which \(\Vert f \Vert_{L^1}\) is replaced by the weaker norm \(\Vert f \Vert_{L^1 + \dot{W}^{-1, n/(n - 1)}}\).
They also proved, by Smirnov's approximation of solenoidal vector charges by elementary solenoids \cite{MR1246427}, that the estimate \eqref{ineqBBVS} was equivalent to an estimate of circulation integrals of vector fields in critical Sobolev spaces \citelist{\cite{MR2075883}\cite{MR2038078}} (see also \cite{MR2500488} for a discussion of the equivalence).
In \cite{MR2078071} \eqref{ineqBBVS} was given a direct proof, based on slicing of the Euclidean space and integration by parts, and that is somehow reminiscent of the proof by Gagliardo and Nirenberg of the limiting Sobolev embedding \cites{MR0109940,MR0102740}. In the planar case \(n = 2\), the estimate is equivalent to the isoperimetric inequality and to the Gagliardo--Nirenberg--Sobolev inequality.

The estimate \eqref{ineqBBVS} has been used to prove certain Gagliardo-Nirenberg inequalities for differential forms on $\mathbb{R}^n$ \citelist{\cite{MR2122730}\cite{MR2293957}}, and has been applied to several complex variables \cite{MR2592736}, to wave and Schr\"odinger equations \cite{MR2876830}, to fluid mechanics and magnetism \cite{MR3439724}, to plasticity \cite{MR2677615}, to list a few (see also the review \cite{MR3298002}).

The compensation phenomenon behind the estimate \eqref{ineqBBVS} turned out to be quite robust and extended quite beyond its original setting: the Sobolev space \(\dot{W}^{1, n} (\mathbb{R}^n)\) can be replaced by a more general fractional Sobolev or Lorentz spaces \citelist{\cite{MR2550188}\cite{MR3054337}}, the divergence-free condition is in fact a particular case of a class of differential conditions \citelist{\cite{MR2293957}\cite{MR2443922}\cite{MR3085095}}, the Euclidean space \(\mathbb{R}^n\) can be replaced by domains, or compact manifolds with smooth boundaries, under suitable boundary conditions \citelist{\cite{MR2293957}\cite{MR2332419}}, and it can also be replaced by homogeneous Lie groups with sub-Riemannian structures \citelist{\cite{MR2511628}\cite{MR3191973}\cite{MR3091819}}. See also \cite{MR3054337,MR3314065,MR2381898} for some related work.

\medbreak

In order to understand the geometric content of \eqref{ineqBBVS}, we propose to study on which geometric structures the estimate \eqref{ineqBBVS} does hold \emph{globally}. In the present work we prove \eqref{ineqBBVS} in the framework of symmetric spaces of non-compact type.

\begin{thm} \label{thm:main}
Let $(\M^m,g)$ be a Riemannian globally symmetric space of non-compact type, of real dimension $m$ and endowed with an invariant Riemannian metric $g$.
If $f$ is a smooth vector field on $\M$ with
$$
\div f = 0,
$$
then for any compactly supported smooth vector field $\phi$ on $\M$, we have
\begin{equation}\label{eq:thm}
\left| \int_{\M} \langle f, \phi \rangle dV \right| \leq C \Vert f\Vert_{L^1(dV)} \Vert \nabla \phi\Vert_{L^m(dV)}.
\end{equation}
\end{thm}

Here we write $\div f$ the divergence of a vector field $f$ on the manifold $\M$ with respect to the metric $g$. Moreover, we write $\langle \cdot, \cdot \rangle$ for the pointwise inner product between two vectors or tensors on $\M$ with respect to $g$, and $dV$ for the volume measure of $\M$ with respect to $g$. We also write $\nabla$ for the Levi-Civita connection on $(\M,g)$, so that $\nabla$ sends a vector field on $\M$ to a $(1,1)$ tensor on $\M$. Finally, we denote $\Vert f\Vert_{L^1(dV)} := \int_{\M} |f| dV$, and $\Vert \nabla \phi\Vert_{L^m(dV)}:= \left( \int_{\M} |\nabla \phi|^m dV \right)^{1/m}$, where $|\cdot|$ denotes the norm of a vector or a tensor on $\M$ with respect to $g$.

One way of saying that the manifold $\M$ is a Riemannian globally symmetric space of non-compact type is to write it as $\M = G/K$, where $G$ is a non-compact connected semisimple real Lie group with finite center, and $K$ is a maximal compact subgroup of $G$. (In general, \(K\) need not be a normal subgroup, and thus \(G/K\) does not have a canonical Lie group structure.) So Theorem~\ref{thm:main} can be thought of as an extension of \eqref{ineqBBVS} to semisimple structures.
Examples of such spaces include the real and complex hyperbolic spaces, the space of positive definite $n \times n$ symmetric matrices with determinant $1$, etc.
The case of the hyperbolic space \(\mathbb{H}^n\) in Theorem~\ref{thm:main} was established by the authors in \cite{CvSY2016} with, as explained below, a different strategy of proof that takes profit of their richer structure.

The proof of Theorem~\ref{thm:main} follows the strategy of \cite{MR2078071} used to prove \eqref{ineqBBVS}.
In the Euclidean case, it is possible to foliate the space by parallel hyperplanes, prove estimates on each, and then integrate with respect to the foliation parameter. This foliation cannot be performed in general on a symmetric space. We prove instead an estimate on a codimension~\(1\) submanifold corresponding to a subgroup in the parametrization of \(\M\) by the Iwasawa decomposition of a group $G$, where $G$ is the identity component of the group of isometries of \(\M\). We then transport the estimate everywhere on \(\M\) by the action of the group \(G\). Thanks to the compactness of \(K\), the resulting quantities are finite and give the estimate.

It remains to explain how the estimate on the codimension \(1\) submanifold is obtained. The idea, following \cite{MR2078071}, is to first split \(\phi\) in a part which is bounded and another whose derivative is bounded. The integral corresponding to the first part is bounded trivially, and the remaining integral is bounded by integration by parts, using the divergence-free condition.
In our setting, the integration by parts relies on the assumption that $\M$ is of non-compact type, which precludes notably the presence of closed geodesics. The splitting argument is essentially the same as in the Euclidean space at smaller scales, where the geometries are comparable; at larger scales the geometry of \(\M\) plays again an important role.

The method in \cite{CvSY2016} there is different from what we have in the general case here: for instance, when $n = 2$, in \cite{CvSY2016} we would first obtain an estimate on all geodesics in $\mathbb{H}^2$, and then average it over all geodesics on $\mathbb{H}^2$. If we restrict our current proof to $\mathbb{H}^2$, we would be first integrating over horocircles in $\mathbb{H}^2$, and then averaging over all horocircles.

Finally we give an application of Theorem~\ref{thm:main} to the Laplace equation on $\M$. Let $\Delta = \div \circ \nabla$ be the Laplace-Beltrami operator on $\M$. It is known that the $L^2$ spectrum of $\M$ is bounded away from zero. Indeed, the $L^2$ spectrum consists of the half-line $(-\infty,-|\rho|^2]$ where $\rho$ is the half-sum of positive roots of an Iwasawa decomposition of $G$ \cite[Section 4.2]{MR1736928}. Thus one can consider negative powers of $-\Delta$, and $(-\Delta)^{-\gamma}$ will be bounded linear operators on $L^2(dV)$ for all $\gamma > 0$.

\begin{thm} \label{thm:Laplace_application}
Suppose $f$ is a smooth vector field on $\M$ with $\div f = 0$. If $\varphi$ is a compactly supported smooth vector field on $\M$ with $\|\varphi\|_{L^{\infty}(dV)}+\|\nabla \varphi\|_{L^{\infty}(dV)} \leq 1$, then the solution $u \in L^2(dV)$ of the equation
$$
\Delta u = \langle f, \varphi \rangle
$$
satisfies
$$
\|\nabla u\|_{L^{\frac{m}{m-1}}(dV)} \leq C \|f\|_{L^1(dV)}.
$$
\end{thm}

If one estimates the solution $u$ by $\|f\|_{L^p(dV)}$ for some $p$ slightly bigger than 1, then one could do so by standard elliptic theory. This fails at the borderline exponent $p=1$. The condition $\div f = 0$ is what makes the conclusion of Theorem~\ref{thm:Laplace_application} possible.
The assumption that $\varphi$ is compact supported is only used to allow for certain integration by parts, and does not appear in any quantitative way in the conclusion of the theorem.

\medskip

The paper will be organized as follows. In Section~\ref{sect:prelim}, we recall some basic notations about symmetric spaces, and some basic facts we will need, such as the Iwasawa decomposition, the Killing metric, the exponential coordinates for a system of vector fields, and some integration formulae on symmetric spaces. The reader will find a more detailed account of many of these facts in the works of Helgason \citelist{\cite{MR1834454}*{Chapters II, III, V, VI, IX}\cite{MR1790156}*{Chapter I}}, or Knapp \cite[Chapter VI]{MR1920389}. In the proof of Theorem~\ref{thm:main}, we need a decomposition lemma for functions on a Lie subgroup of the isometry group of $\M$; we isolate this in Section~\ref{sect:decomp}. The proof of Theorem~\ref{thm:main} is then given in Section~\ref{sect:proof}, and that of Theorem~\ref{thm:Laplace_application} in Section~\ref{sect:applications}.

\medskip

\noindent \textbf{Acknowledgements}. S.C. was partially supported by NSF grant DMS-1201474. J.V.S. was partially supported by the Fonds de la Recherche Scientifique-FNRS. P.-L.Y. was  partially supported by the Early Career Grant CUHK24300915 from the Hong Kong Research Grant Council. It is also our pleasure to thank Roe Goodman for pointing out for us an argument in Section~\ref{sect:expcoords}.

\section{Notations and Preliminaries} \label{sect:prelim}

\subsection{Assumptions} \label{sect:assump}

Let $(\M^m,g)$ be a Riemannian globally symmetric space. In other words, $\M$ is a connected Riemannian manifold, and at every point of $\M$, the geodesic reflection symmetry defines a global isometry of $\M$. Throughout this paper, we will fix a point $x_0 \in \M$. Write $G=I_0(\M)$ for the identity component of the isometry group of $\M$ (which we think of as acting on $\M$ on the left), and $K$ for the subgroup of $G$ that leaves $x_0$ fixed.
Then $\M$ is diffeomorphic to $G/K$ via the identification $g x_0 \simeq gK$.
We will write $s_{x_0}$ for the geodesic reflection symmetry at $x_0$, and $\sigma \colon G \to G$ for the involutive automorphism of $G$ given by
\begin{equation} \label{eq:sigma_def}
\sigma (g) = s_{x_0} g s_{x_0}, \quad g \in G.
\end{equation}
Let $\g$ and $\k$ be the Lie algebras of $G$ and $K$ respectively, and write $$\theta = (d\sigma)_e \colon \g \to \g$$ where $e$ is the identity of $G$. Then $\theta$ is an involutive automorphism of $\g$, and
$
\k = \{X \in \g \colon \theta X = X \}.
$
If
$
\p = \{X \in \g \colon \theta X = -X\},
$
then $\g$ is the direct sum
$$
\g = \k \oplus \p.
$$
Let $\pi \colon G \to G/K = \M$ be the natural projection map given by $\pi(g) = gK (\simeq gx_0)$,
and $(d\pi)_e \colon \g \to T_{x_0} \M.$
Then $(d\pi)_e \k = \{0\}$, and $(d\pi)_e$ maps $\p$ isomorphically onto $T_{x_0} \M$. The latter allows one to identify $T_{x_0}\M$ with $\p$. See \cite{MR1834454}*{Chapter IV, Theorem 3.3}.

In what follows we assume that $\M$ is of non-compact type. In other words, we assume that
\begin{enumerate}[(i)]
\item $\g$ is semisimple,
\item $\g$ is non-compact, and
\item the decomposition $\g = \k \oplus \p$ is a Cartan decomposition of $\g$.
\end{enumerate}

Condition (i) means that
the Killing form $B \colon \g \to \g$, defined by
\begin{equation} \label{eq:Killingformdef}
B(X,Y) = \operatorname{trace}\, ( \ad_{\g} X \circ \ad_{\g} Y )
\end{equation}
for $X, Y \in \g$, is non-degenerate on $\g$.
Condition (ii) means that
 the adjoint group $\operatorname{Int}(\g)$ of $\g$ is not compact; here the adjoint group $\operatorname{Int}(\g)$ is the connected Lie subgroup of $GL(\g)$ whose Lie algebra is $\ad(\g) := \{\ad_{\g} X \colon X \in \g\}$. Under the semisimple assumption in (i), this is equivalent
to saying that the Killing form $B$ is not strictly negative definite on $\g$.
Condition (iii) means that
$\theta \colon \g \to \g$ is a Lie algebra automorphism of $\g$, and that
the symmetric bilinear form
$$
B_{\theta}(X,Y) := -B(X,\theta Y)
$$
is positive definite on $\g$. In particular, the restriction of $B$ to $\k$ is negative definite, and the restriction of $B$ to $\p$ is positive definite.

From the assumption that $\M$ is of non-compact type, one can show that the map $\pi \circ \exp \colon \g \to G/K$ restricts to a diffeomorphism from $\p$ to $\M$; in particular, $\M$ is simply connected. See e.g. \cite{MR1834454}*{Chapter VI, Theorem 1.1}.

Another way of saying that $\M$ is a Riemannian globally symmetric space of non-compact type is to say that $\M = G/K$, where $G$ is a non-compact connected semisimple real Lie group with finite center and $K$ is a maximal compact subgroup of $G$.
We will not need to use this characterization in this paper.

We will denote by $r$ the rank of $\M$. In other words,
the maximal dimension of a flat, totally geodesic submanifold of $\M$ is~$r$.

\subsection{Iwasawa decompositions} \label{sect:Iwasawa}

To prove our main Theorem~\ref{thm:main}, we will use a family of Iwasawa decompositions of the semisimple Lie group $G = I_0(\M)$. Recall the Cartan decomposition $\g = \k \oplus \p$ of the real semi-simple Lie algebra $\g$. Let $\a$ be a maximal abelian subalgebra of $\g$ inside $\p$; the dimension of $\a$ is then $r$, since $\M$ has rank $r$. Let $\a^*$ be the set of all real linear functionals on $\a$. For each $\alpha \in \a^*$, let
$$
\g_{\alpha} := \bigl\{ X \in \g \colon \ad_{\g}(H)X = \alpha(H) X \text{ for all $H \in \a$} \bigr\}
$$
Then \begin{equation} \label{eq:comm_g_alpha}
[\g_{\alpha_1}, \g_{\alpha_2}] \subseteq \g_{\alpha_1 + \alpha_2}
\end{equation}
for any $\alpha_1, \alpha_2 \in \a^*$. Since $\{\ad_{\g}(H) \colon H \in \a\}$ is a family of commuting linear endomorphisms on $\g$, they can be simultaneously diagonalized, which gives
\begin{equation} \label{eq:sim_diag_g_alpha}
\g = \g_0 \oplus \bigoplus_{\alpha \in \Sigma} \g_{\alpha}
\end{equation}
where $\Sigma$ be the set of restricted roots of $\a$, i.e. the set of all non-zero $\alpha \in \a^*$ for which $\g_{\alpha} \ne \{0\}$. Let $\Sigma^+$ be a choice of positive roots of $\Sigma$, and $\n$ be the nilpotent Lie algebra given as the sum of the root spaces of $\Sigma^+$, i.e.
$$
\n := \bigoplus_{\alpha \in \Sigma^+} \g_{\alpha}.
$$
Then since $\g$ is real and semi-simple, we have the following Iwasawa decomposition of $\g$:
$$\g = \k \oplus \a \oplus \n.$$
Let $K$, $A$, $N$ be the Lie subgroups of $G$ generated by $\k$, $\a$ and $\n$. Then since $G$ is a connected semi-simple real Lie group, we have the following Iwasawa decomposition of $G$:
$$G = KAN,$$ in the sense that
\begin{align*}
K \times A \times N &\to G \\
(k,a,n) &\mapsto kan
\end{align*}
is a diffeomorphism of $K \times A \times N$ onto $G$. Since $g \mapsto g^{-1}$ is a diffeomorphism of $G$, it also follows that $G = NAK$. This allows one to identify $\M = G/K$ with $N \times A$, via the map
\begin{align}
N \times A &\to G/K = \M  \label{eq:NA=M}\\
(n,a) &\mapsto naK. \notag
\end{align}

In addition, since $\Ad a$ preserves $N$ for each $a \in A$ (in fact $\Ad a$ preserves $\g_{\alpha}$ for all $\alpha \in \Sigma$), it follows that $AN = NA$ is a Lie subgroup of $G$, and that the Iwasawa decomposition can be written $$G = ANK$$ as well.

The above Iwasawa decomposition of $G$ depends on the choice of a maximally abelian subalgebra $\a$ of $\g$ inside $\p$, and the choice of positive roots $\Sigma^+$ in $\Sigma$. In the proof of Theorem~\ref{thm:main}, we will need to use a family of different Iwasawa decompositions of $G$, and obtain uniform estimates for such. 
In particular, the constant $C$ in Lemma~\ref{lem:decompS'unif} needs to be independent of the choice of $\a$ and $\Sigma^+$ (or in other words, independent of the choice of $A$ and $N$ in the Iwasawa decomposition $KAN$ of $G$). This may be explained via the fact that all Iwasawa decompositions are conjugate to each other; for the latter fact, see for instance  \cite{MR1476489}*{Remark after Theorem 4.6}.

Note that $x_0 \in \M$ is fixed throughout our paper, and so is $K$ the stabilizer of $x_0$ in $G = I_0(\M)$. So we will just say that something is independent of the choice of Iwasawa decomposition $KAN$ of $G$, if it is independent of the choice of $A$ and $N$.

\subsection{The Killing form}

Let $G = I_0(\M)$, $\g$ be its Lie algebra. Recall the Killing form $B \colon \g \to \g$ (defined as in (\ref{eq:Killingformdef})). Let us note here two standard facts about $B$, that we will use later on. First,
\begin{equation} \label{eq:Kill_theta}
B(\theta X, \theta Y) = B(X,Y) \quad \text{for all $X, Y \in \g$}.
\end{equation}
This holds because $\theta$ is an automorphism of $\g$. Indeed, we have $$\ad_{\g} (\theta X) = \theta  \circ (\ad_{\g} X) \circ \theta^{-1}$$ for all $X \in \g$, so for any $X, Y \in \g$, we have
$$B(\theta X, \theta Y) = \operatorname{trace}( \theta \circ (\ad_{\g} X) \circ (\ad_{\g} Y) \circ \theta^{-1} ) = B(X,Y).$$
Also, the Killing form $B$ satisfies
\begin{equation} \label{eq:Kill_sym}
B([Z,X],Y) + B(X,[Z,Y]) = 0 \quad \text{for all $X, Y, Z \in \g$.}
\end{equation}
This holds because the Bianchi identity implies that
$$
\ad_{\g}[X,Y] = [\ad_{\g} X, \ad_{\g} Y] \quad \text{for all $X, Y \in \g$.}
$$

\subsection{The Killing metric} \label{sect:Kill_metric}

Recall $G = I_0(\M)$, the identity component of the isometry group of $(\M,g)$. The Riemannian metric $g$ on $\M$ is of course $G$-invariant. However, in general, \(\M\) can be endowed with infinitely many non-proportional \(G\)-invariant metrics.
We will now introduce one particularly convenient such metric $g_0$. It will be defined by a suitable restriction of the Killing form $B$ of $\g$, and we will call it the Killing metric on $\M$.

Indeed, recall we fixed, once and for all, a point $x_0 \in \M$. Let $\theta = (d\sigma)_e \colon \g \to \g$ be the involutive automorphism of $\g$, where $\sigma \colon G \to G$ is the automorphism of $G$ defined by (\ref{eq:sigma_def}), and let $\g = \k \oplus \p$ be the associated Cartan decomposition. Also recall the natural projection map $\pi \colon G \to G/K = \M$ given by $\pi(x) = gK = gx_0$. Then the restriction of $(d\pi)_e$ to $\p$ provides a linear isomorphism between $\p$ and $T_{x_0}\M$. Now we will define $g_0$ at $x_0$, by restricting the Killing form \(B\) to \(\p \times \p\) \cite{MR1834454}*{Chapter V, Section 5}. More precisely, given \(X, Y \in T_{x_0} \M\), we identify $X$, $Y$ with the corresponding vectors in $\p$ via this isomorphism, and define $g_0$ at $x_0$ by
\[
  g_0 (X, Y) = B (X, Y).
\]
Since the Killing form $B$ restricts to a positive definite symmetric bilinear form on $\p$, \(g_0\) at \(x_0\) is positive definite.
In order to extend this metric to the whole manifold \(M\) by the action of \(G\), we observe that \(g_0\) at \(x_0\) is invariant under the action of \(K\).
Indeed, it suffices to show that
$$
\left. \frac{d}{dt} \right|_{t=0} g_0( dL_{e^{tZ}} X, dL_{e^{tZ}} Y ) = 0
$$
whenever $Z \in \k$ and $X, Y \in T_{x_0}\M$, where $dL_g$ is the differential of the left action of $g$ on $\M$. To see this, it will be convenient, for a moment, not to identify $\p$ with $T_{x_0}\M$ via $(d\pi)_e$. So if $X,Y \in T_{x_0}\M$, let $X = (d\pi)_e \tilde{X}$ and $Y = (d\pi)_e \tilde{Y}$, where $\tilde{X}, \tilde{Y} \in \p$. Then for $Z \in \k$, $dL_{e^{tZ}}X$ is given by \begin{equation*}
\begin{split}
dL_{e^{tZ}}X &= \left. \frac{d}{ds} \right|_{s=0} e^{tZ} e^{s \tilde{X}} x_0 \\
 &= \left. \frac{d}{ds} \right|_{s=0}  e^{tZ} e^{s \tilde{X}} e^{-tZ} x_0
 = \left. \frac{d}{ds} \right|_{s=0} \pi(e^{tZ} e^{ s \tilde{X}} e^{-tZ}),
\end{split}
 \end{equation*}
 so $dL_{e^{tZ}}X = (d\pi)_e(\Ad(e^{tZ}) \tilde{X}) = (d\pi)_e(e^{t\ad Z} \tilde{X})$. Similarly $dL_{e^{tZ}}Y = (d\pi)_e(e^{t\ad Z} \tilde{Y})$. Thus
 $$
 \left. \frac{d}{dt} \right|_{t=0} dL_{e^{tZ}}X = (d\pi)_e[Z,\tilde{X}], \quad \text{and} \quad \left. \frac{d}{dt} \right|_{t=0} dL_{e^{tZ}}Y = (d\pi)_e[Z,\tilde{Y}].
 $$
 Since $[Z,\tilde{X}], [Z,\tilde{Y}] \in \p$, it follows that
\begin{multline*}
\left. \frac{d}{dt} \right|_{t=0} g_0( dL_{e^{tZ}} X, dL_{e^{tZ}} Y )\\
= g_0((d\pi)_e [Z,\tilde{X}],(d\pi)_e \tilde{Y})+ g_0((d\pi)_e \tilde{X},(d\pi)_e [Z,\tilde{Y}])\\
= B([Z,\tilde{X}],\tilde{Y}) + B(\tilde{X},[Z,\tilde{Y}]) = 0,
\end{multline*}
the last equality following from \eqref{eq:Kill_sym}.
This shows that \(g_0\) is invariant at \(x_0\) under the action of the group of isometries \(K\), and hence it can be extended to a unique \(G\)-invariant \(g_0\) metric on the whole space \(M \simeq G/K\).

Suppose now $G = KAN$ is an Iwasawa decomposition of $G$ as in the last subsection. We may then identify $\M$ with a Lie group, and calculate $g_0$ using this identification, as follows. Let $$S = AN = NA = \{na \in G \colon a \in A, n \in N\}.$$ Then $S$ is a closed subgroup of $G$, and the map $$\pi \colon G \to G/K \simeq \M$$ restricts to a diffeomorphism between $S$ and $\M$. One can thus identify $\M$ with the Lie group $S$ via this diffeomorphism. Upon this identification, every element in the Lie algebra $\s = \a \oplus \n$ of $S$ defines an $S$-invariant vector field on $\M$.
We claim now, that under this identification, we have for every \(X, Y \in \s\),
\begin{equation}
\label{eqKillingS}
 g_0 (X,Y) = B\left( \tfrac{X - \theta X}{2} , \tfrac{Y - \theta Y}{2} \right).
\end{equation}
Indeed, it suffices to check that $g_0$ at $x_0$ is given by the expression on the right. But $X \in \s$ corresponds to $(d\pi)_e X$ in $T_{x_0}\M$ under our identification of $S$ with $\M$, and $(d\pi)_e X$ corresponds to  $\frac{X - \theta X}{2}$ under our identification of $T_{x_0}\M$ with $\p$ (the latter following from that $$(d\pi)_e X = (d\pi)_e \left( \frac{X - \theta X}{2} \right), $$ i.e. $(d\pi)_e  \left( \frac{X + \theta X}{2} \right) = 0$, which in turn is a consequence of the fact that $\frac{X + \theta X}{2} \in \k$, the $1$-eigenspace of $\theta$).
So we see that \eqref{eqKillingS} holds,
regardless of which Iwasawa decomposition of $G$ we used to define $S$. By (\ref{eq:Kill_theta}), equation \eqref{eqKillingS} can also be written as
$$
g_0(X,Y) = \frac{B(X,Y)-B(\theta X, Y)}{2}
$$
for all $X,Y \in \s$.

The Killing metric $g_0$ may not be the same $G$-invariant Riemmanian metric $g$ we had on $\M$ nor be proportional to it; nevertheless, this is a convenient metric to use
for computations, for they enjoy certain orthogonality relations. Suppose $G = KAN$ is an Iwasawa decomposition of $G$, and $S = AN$ with Lie algebra $\s$. Then
\begin{equation} \label{eq:sortho}
\s = \a \oplus \n = \a \oplus \bigoplus_{\alpha \in \Sigma^+} \g_{\alpha}.
\end{equation}
Upon identifying $\M$ with $S$, so that the Killing metric $g_0$ defines a left-invariant metric on $S$, we claim that the decomposition \eqref{eq:sortho} is an orthogonal decomposition with respect to $g_0$. In other words, we have that
\begin{enumerate}[(a)]
\item \label{item:ortho1} $\a$ is orthogonal to $\g_{\alpha}$ with respect to $g_0$, for all $\alpha \in \Sigma^+$; and
\item \label{item:ortho2} $\g_{\alpha}$ is orthogonal to $\g_{\beta}$ with respect to $g_0$, whenever $\alpha, \beta \in \Sigma^+$ and $\alpha \ne \beta$.
\end{enumerate}

To see \eqref{item:ortho1}, let $H \in \a$, $X \in \g_{\alpha}$ with $\alpha \in \Sigma^+$. Then
$$
g_0(H,X) = \frac{B(H,X) - B(\theta H, X)}{2} = B(H,X) = 0,
$$
the second last equality following from $H \in \a \subset \p$ (so that $\theta H = -H$), and the last following from \eqref{eq:sim_diag_g_alpha} and \eqref{eq:comm_g_alpha}: indeed the two combined shows that there exists a basis of $\g$, with respect to which $\ad_{\g} H$ is represented by a diagonal matrix, and $\ad_{\g} X$ is represented by a lower triangular matrix. Thus $B(H,X) = \operatorname{trace} (\ad_{\g} H \circ \ad_{\g} X) = 0$.

To see \eqref{item:ortho2}, let $\alpha, \beta \in \Sigma^+$ with $\alpha \ne \beta$. Let $X \in \g_{\alpha}$, $Y \in \g_{\beta}$. Then
$$
g_0(X,Y) = \frac{B(X,Y) - B(\theta X, Y)}{2}.
$$
But $B(X,Y) = 0$, since as above, there exists a basis of $\g$, with respect to which both $\ad_{\g} X$ and $\ad_{\g} Y$ are represented by a lower triangular matrix. Also, $\theta X \in \g_{-\alpha}$, from which it follows that $\ad_{\g} (\theta X) \circ \ad_{\g} Y$ maps $\g_{\gamma}$ into $\g_{\gamma+\beta-\alpha}$ for all $\gamma \in \Sigma \cup \{0\}$. Since $\beta - \alpha \ne 0$, it follows that with respect to a suitable basis, $\ad_{\g} (\theta X) \circ \ad_{\g} Y$ is represented by a matrix which is zero on the diagonal, so $B(\theta X, Y) = 0$ as well, and this gives $g_0(X,Y) = 0$.

From the above orthogonality relations, it follows that we can find a basis $H_1, \dots, H_r$ of $\a$, and a basis $Y_1, \dots, Y_{m-r}$ of $\n = \bigoplus_{\alpha \in \Sigma^+} \g_{\alpha}$, such that each $Y_j \in \g_{\alpha_j}$ for some $\alpha_j \in \Sigma^+$, and such that the $H_i$'s and the $Y_j$'s together form an orthonormal basis of $\s = \a \oplus \n$ with respect to $g_0$. Such a basis will be called a \emph{good basis} of $\s$; we will also identify such with a frame of $S$-invariant vector fields on $\M$.
A good basis of $\s$ enjoys good orthogonality and commutativity conditions (they are orthonormal with respect to $g_0$, and they respect the commutator conditions given by \eqref{eq:comm_g_alpha}). As such they are particularly convenient for computations. We will also use it to help us formulate the uniform decomposition lemma in Section~\ref{sect:decomp}.

We note in passing that the restrictions of $g_0$ to $\a$ and $\n$ also yields positive definite inner products on $\a$ and $\n$, which we think of as left-invariant Riemannian metrics on $A$ and $N$ respectively. By abuse of notation, we will also call these the Killing metrics on $A$ and $N$, and denote them by the same symbol $g_0$.

Finally, we claim that it suffices to establish Theorem~\ref{thm:main} for any one choice of $G$-invariant metric on $\M$, for then the same conclusion will follow for all other $G$-invariant metrics on $\M$ (proof to follow in the next paragraph). So in proving Theorem~\ref{thm:main}, we will assume, without loss of generality, that \(g = g_0\) the Killing metric. This will allow us to take advantage of the orthogonality conditions of $g_0$ laid out earlier.

To verify our claim above, suppose we have proved Theorem~\ref{thm:main} for some $G$-invariant metric $g$ on $\M$, and suppose $\tilde{g}$ is another $G$-invariant metric on $\M$. If $f$ is a vector field on $\M$ with
$$
\textrm{div}_{\tilde{g}} f = 0,
$$ 
(here we use the subscript $\tilde{g}$ to denote the dependence on the metric), then we also have
\begin{equation} \label{eq:div_g_f_equals_0}
\textrm{div}_{g} f = 0.
\end{equation} 
To see this, note that 
$$
\textrm{div}_{g} f \, dV_g = \mathcal{L}_f (dV_g)
$$
where $\mathcal{L}_f$ denotes the Lie derivative with respect to the vector field $f$; this is a consequence of the Cartan formula for Lie derivatives of differential forms. Now $dV_g$ is just a multiple of $dV_{\tilde{g}}$, since both $dV_g$ and $dV_{\tilde{g}}$ are G-invariant volume forms, and the space of G-invariant volume forms is one-dimensional.
Also, 
$$
\mathcal{L}_f (dV_{\tilde{g}}) = \textrm{div}_{\tilde{g}} f \,  dV_{\tilde{g}} = 0 
$$
by assumption. Thus (\ref{eq:div_g_f_equals_0}) is verified.
Next, if $\phi$ is any smooth vector field on $\M$, we will find another smooth vector field $\tilde{\phi}$ on $\M$, such that
\begin{equation}  \label{eq:change_metric_corrected}
\langle f, \phi \rangle_{\tilde{g}} = \langle f, \tilde{\phi} \rangle_g.
\end{equation}
Indeed, if $\Phi$ is the $1$-form on $\M$ that one obtains by lowering the indices with the metric $\tilde{g}$, then $\langle f, \phi \rangle_{\tilde{g}} = \Phi(f)$; now we raise the indices with the metric $g$, we obtain from $\Phi$ a vector field $\tilde{\phi}$ that satisfies $\Phi(f) = \langle f, \tilde{\phi} \rangle_g$. Together we have (\ref{eq:change_metric_corrected}); indeed, in any local coordinates, we have
$$
\tilde{\phi}^i = g^{ij} \tilde{g}_{jk} \phi^k.
$$
Finally, from (\ref{eq:change_metric_corrected}), we have
$$
\left| \int_{\M} \langle f, \phi \rangle_{\tilde{g}} dV_{\tilde{g}} \right| = C \left| \int_{\M} \langle f, \tilde{\phi} \rangle_{g} dV_{g} \right|
$$
since $dV_{\tilde{g}}$ is a constant multiple of $dV_g$. Since $\textrm{div}_g f = 0,$ the right hand side can now be estimated with Theorem~\ref{thm:main} for $g$. 
Note that $$\Vert f\Vert_{L^1(dV_g)} = \int_{\M} |f|_g dV_g \simeq \int_{\M} |\tilde{f}|_{\tilde{g}} dV_{\tilde{g}} = \Vert \tilde{f}\Vert_{L^1(dV_{\tilde{g}})},$$ and
$$
|\nabla_g \tilde{\phi}|_g \lesssim |\nabla_g \phi|_g + |\phi|_g \simeq |\nabla_{\tilde{g}} \phi|_{\tilde{g}} + |\phi|_{\tilde{g}}
$$
(the latter following from that $\nabla_g = \nabla_{\tilde{g}}$; see \cite{MR1834454}, Chapter IV, Section 4, Corollary 4.3, or \cite{MR1809879}, Chapter I, Proposition 1.2.1), from which we obtain
\begin{align*}
\Vert \nabla_g \tilde{\phi} \Vert_{L^m(dV_g)} 
&= \left( \int_{\M} |\nabla_g \tilde{\phi}|_g^m dV_g \right)^{1/m} \\
& \lesssim \left( \int_{\M} |\nabla_{\tilde{g}} \phi|_{\tilde{g}}^m + |\phi|_{\tilde{g}}^m dV_{\tilde{g}} \right)^{1/m} \\
&\simeq \Vert \nabla_{\tilde{g}} \phi \Vert_{L^m(dV_{\tilde{g}})} + \|\phi\|_{L^m(dV_{\tilde{g}})}.
\end{align*}
We now estimate
$$
\|\phi\|_{L^m(dV_{\tilde{g}})} \lesssim \Vert \nabla_{\tilde{g}} \phi \Vert_{L^m(dV_{\tilde{g}})};
$$
this inequality will be established in Lemma~\ref{lem:Hardy} in the sequel.
Thus altogether, Theorem~\ref{thm:main} for $g$ gives that
$$
\left| \int_{\M} \langle f, \phi \rangle_{\tilde{g}} dV_{\tilde{g}} \right| 
\leq C \|f\|_{L^1(dV_{\tilde{g}})} \|\nabla_{\tilde{g}} \phi \|_{L^m(dV_{\tilde{g}})}
$$
as well, as desired.

\subsection{Computation of the divergence} \label{sect:div}

We now proceed to carry out some explicit calculations, for the divergence of a vector field $f$ on $\M$; this will help us make some computation more concrete in the proof of Theorem~\ref{thm:main}. We will identify $\M$ with a Lie group $S$ as before, and perform the computations in terms of a basis of left-invariant vector fields on $S$.

To this end, let $G = I_0(\M)$, $K$ be the stabilizer of $x_0 \in \M$, and $G = KAN$ be a corresponding Iwasawa decomposition. Let $S = AN$. As before, we identify $\M$ with $S$, and identify the Lie algebra $\s$ of $S$ with the space of $S$-invariant vector fields on $\M$. Note that $\s = \a \oplus \n$, where $\n = \bigoplus_{\alpha \in \Sigma^+} \g_{\alpha}$.

Suppose now we are given a basis $H_1, \dots, H_r$ of $\a$, and a basis $Y_1, \dots, Y_{m-r}$ of $\n$ such that each $Y_j \in \g_{\alpha_j}$ for some $\alpha_j \in \Sigma_+$. For convenience, we will write
\begin{equation} \label{eq:X_label}
\begin{cases}
X_1 = H_1, \quad \dotsc, \quad  X_r = H_r, \\
X_{r+1} = Y_1, \quad \dotsc, \quad X_m = Y_{m-r}
\end{cases}
\end{equation}
so that $\{X_1,\dots,X_m\}$ is a basis of $\s$. We think of it as a basis of $S$-invariant vector field on $\M$. A smooth vector field $f$ on $\M$ can then be expressed as a linear combination of $X_1, \dots, X_m$ (with variable coefficients):
$$
f = \sum_{\ell=1}^m f^{\ell} X_{\ell}
$$
where $f^1, \dots, f^m$ are smooth functions on $\M$. We may then compute the divergence of $f$ on $\M$ as follows.

First,
$$
\div f = \sum_{\ell=1}^m \left( X_{\ell} f^{\ell} + f^{\ell} \div X_{\ell} \right),
$$
so we need only compute $\div X_{\ell}$ for $1 \leq \ell \leq m$.
But $\div X_{\ell}$ is given by
\[
 \mathcal{L}_{X_{\ell}} \omega = (\operatorname{div} X_{\ell})\omega,
\]
where $\omega$ is the volume form of $(\M,g)$ with respect to $g$, and \(\mathcal{L}_{X_{\ell}} \omega\) denotes the \emph{Lie derivative} of \(\omega\) with respect to the vector field \({X_{\ell}}\). To compute $\mathcal{L}_{X_{\ell}} \omega$, note that for any vector field $Z$ on $\M$, we have, by the Leibnitz rule for contraction, that
\begin{equation} \label{eq:Lie_der}
\begin{split}
 &(\mathcal{L}_Z \omega) (H_1, \dotsc, H_r, Y_1, \dotsc, Y_{m - r}) \\
 &= \mathcal{L}_Z (\omega (H_1, \dotsc, H_r, Y_1, \dotsc, Y_{m - r})) \\
 &\quad- \omega (\mathcal{L}_Z H_1, \dotsc, H_r, Y_1, \dotsc, Y_{m - r}))\\
 &\qquad- \dotsb - \omega ( H_1, \dotsc, \mathcal{L}_Z H_r, Y_1, \dotsc, Y_{m - r})\\
 &\quad- \omega ( H_1, \dotsc, H_r, \mathcal{L}_Z  Y_1, \dotsc, Y_{m - r})\\
 &\qquad- \dotsb - \omega ( H_1, \dotsc, H_r, Y_1, \dotsc, \mathcal{L}_Z Y_{m - r}).
\end{split}
\end{equation}
Also, the first term on the right hand side is always zero, since both $\omega$ and $H_1, \dots, H_r, Y_1, \dots, Y_{m-r}$ are $S$-invariant.
We now apply \eqref{eq:Lie_der} with \(Z = H_\ell\), for \(\ell \in \{1, \dotsc, r\}\). Note that for \(i \in \{1, \dotsc, r\}\),
\[
 \mathcal{L}_{H_\ell} H_i = [H_\ell, H_i] = 0,
\]
and for \(j \in \{1, \dotsc, m - r\}\),
\[
  \mathcal{L}_{H_\ell} Y_j = [H_\ell, Y_j] = \alpha_j (H_\ell) Y_j.
\]
We thus conclude from \eqref{eq:Lie_der} that
\[
 \mathcal{L}_{H_\ell} \omega = -\sum_{j=1}^{m - r} \alpha_j (H_\ell) \omega.
\]
It follows that
\[
 \operatorname{div} H_\ell = -2 \rho(H_\ell),
\]
where
$\rho$ is the half sum of positive roots, given by
\begin{equation} \label{eq:rhodef}
\rho:= \frac{1}{2} \sum_{\alpha \in \Sigma^+} \alpha,
\end{equation}
Next we apply \eqref{eq:Lie_der} with $Z = Y_{\ell}$, for \(\ell \in \{1, \dotsc, m - r\}\). Note that for each  \(i \in \{1, \dotsc, r\}\),
we have
\[
 \mathcal{L}_{Y_\ell} H_i = [Y_\ell, H_i] = -\alpha_\ell (H_i) Y_\ell,
\]
and thus, by antisymmetry,
\[
 \omega ( H_1, \dotsc, \mathcal{L}_{Y_\ell} H_i, \dotsc, H_r, Y_1, \dotsc, Y_{m - r}) = 0,
\]
On the other hand, if \(j \in \{1, \dotsc, m - r\}\), then
\[
 \mathcal{L}_{Y_\ell} Y_j = [Y_\ell, Y_j],
\]
which can be written as a linear combination of \(Y_1, \dotsc, Y_{j - 1}, Y_{j + 1}, \dotsc, Y_{m - r}\), and thus
\[
 \omega ( H_1, \dotsc, H_r, Y_1, \dotsc, \mathcal{L}_{Y_\ell} Y_j, \dotsc, Y_{m - r})= 0.
\]
Hence we conclude from \eqref{eq:Lie_der} that $\mathcal{L}_{Y_\ell} \omega = 0$. It follows that
\[
 \operatorname{div} Y_j = 0.
\]
Together we see that
\begin{equation} \label{eq:divf_explicit}
\div f = -\sum_{i=1}^r 2\rho(H_i) f^i + \sum_{\ell=1}^m X_{\ell}(f^{\ell}), \quad \text{if $f = \sum_{\ell=1}^m f^{\ell} X_{\ell}$}.
\end{equation}

\subsection{Covariant derivatives of a frame} \label{sect:covar}

Let $G = I_0(\M)$, $K$ be the stabilizer of a point $x_0 \in \M$, and $G = KAN$ be a corresponding Iwasawa decomposition. Let $H_1, \dots, H_r, Y_1, \dots, Y_{m-r}$ be a good basis of $\s = \a \oplus \n$, as defined in Section~\ref{sect:Kill_metric}. In particular, each $Y_j \in \g_{\alpha_j}$ for some $\alpha_j \in \Sigma^+$.
For convenience, we will again label these $H_i$'s and $Y_j$'s as $X_1, \dots, X_m$ as in \eqref{eq:X_label}, and we think of them as a basis of $S$-invariant vector fields on $\M$. Suppose now $H \in \a$ (again identified with an $S$-invariant vector field on $\M$). For later purposes (c.f. Lemma~\ref{lem:Hardy} below), we will prove now that
\begin{equation} \label{eq:parallel}
\nabla_{H} X_{\ell} = 0  \quad \text{for $1 \leq \ell \leq m$}.
\end{equation}

Here $\nabla$ is the Levi-Civita connection with respect to any $G$-invariant metric on $\M$ (as explained before, the Levi-Civita connection is actually independent of which $G$-invariant metric we put on $\M$). In particular, we will compute using the Killing metric $g_0$ on $\M$.
So for any smooth vector fields $X, Y, Z$ on $\M$, we have
\begin{align*}
g_0(\nabla_X Y, Z) = \frac{1}{2} ( & X[g_0(Y,Z)] + Y[g_0(X,Z)] - Z[g_0(X,Y)] \\
& \quad  + g_0([X,Y],Z) - g_0([X,Z],Y) - g_0([Y,Z],X) ).
\end{align*}
If in addition, $X, Y, Z$ are $S$-invariant, then by the left-invariance of the metric $g_0$, the first three terms above are zero, and we get
$$
g_0(\nabla_X Y, Z) = \frac{1}{2} \left( g_0([X,Y],Z) - g_0([X,Z],Y) - g_0([Y,Z],X) \right).
$$

We are now ready to prove \eqref{eq:parallel}. Indeed, fix $H \in \a$, and $1 \leq i \leq r$. We will first show that
\begin{equation} \label{eq:parallelH}
\nabla_{H} H_i = 0,
\end{equation}
by showing that
$$
g_0(\nabla_{H} H_i, H_{\ell}) = g_0(\nabla_{H} H_i, Y_j) = 0
$$
for all $1 \leq {\ell} \leq r$, $1 \leq j \leq m-r$.
But for $1 \leq {\ell} \leq r$, we have
\[
\begin{split}
g_0(\nabla_{H} H_i, H_{\ell}) &= \frac{1}{2} \left(g_0([H,H_i],H_{\ell}) - g_0([H, H_{\ell}], H_i) - g_0([H_i,H_{\ell}], H) \right) \\
&= 0.
\end{split}\]
Next, for $1 \leq j \leq m-r$, we have
\[
\begin{split}
g_0(\nabla_{H} H_i, Y_j) &= \frac{1}{2} \left(g_0([H,H_i],Y_j) - g_0([H, Y_j], H_i) - g_0([H_i,Y_j], H) \right) \\
&= 0,
\end{split}
\]
where in the last equality we used that $[H,Y_j] = \alpha_j(H) Y_j$ is orthogonal to $H_i$ with respect to $g_0$, and similarly that $[H_i,Y_j]$ is orthogonal to $H$ with respect to $g_0$. This proves \eqref{eq:parallelH}.

Next, fix $H \in \a$, and $1 \leq j \leq r$. We will show that
\begin{equation} \label{eq:parallelX}
\nabla_{H} Y_j = 0,
\end{equation}
by showing that
$$
g_0(\nabla_{H} Y_j, H_i) = g_0(\nabla_{H} Y_j, Y_{\ell}) = 0
$$
for all $1 \leq i \leq r$, $1 \leq {\ell} \leq m-r$.
But for $1 \leq i \leq r$, we have
\begin{equation*}
\begin{split}
g_0(\nabla_{H} Y_j, H_i)
&= \frac{1}{2} \left(g_0([H,Y_j], H_i) - g_0([H, H_i], Y_j) - g_0([Y_j,H_i], H) \right) \\
&= \frac{1}{2} \left(\alpha_j(H) g_0(Y_j,H_i) - 0 + \alpha_j(H_i) g_0(Y_j,H)  \right) \\
&= 0.
\end{split}
\end{equation*}
Here we used that $Y_j$ is orthogonal to $H_i$ and $H$ with respect to $g_0$.
Similarly, for $1 \leq {\ell} \leq m-r$, we have
\[
\begin{split}
g_0(\nabla_{H} Y_j, Y_{\ell})
&= \frac{1}{2} \left(g_0([H,Y_j],Y_{\ell}) - g_0([H, Y_{\ell}], Y_j) - g_0([Y_j,Y_{\ell}], H) \right) \\
&= \frac{1}{2} \left(\alpha_j(H) \delta_{j\ell} - \alpha_{\ell}(H) \delta_{j\ell} - 0  \right) \\
&= 0,
\end{split}
\]
the second-to-last equality following since $[Y_j,Y_{\ell}] \in \g_{\alpha_j+\alpha_{\ell}} \subset \n$ whereas $H \in \a$.
The two identities together verify \eqref{eq:parallelX}. \eqref{eq:parallel} then follows from \eqref{eq:parallelH} and \eqref{eq:parallelX}.

\subsection{Exponential coordinates} \label{sect:expcoords}

Let $G = I_0(\M)$, and $G = KAN$ be an Iwasawa decomposition of $G$. Another fact we will need is that $N$ is always \emph{graded}. This means that the Lie algebra $\n$ of $N$ is a direct sum of subspaces
$$
\n = \bigoplus_{k=1}^{s} V_k,
$$
and $[V_j,V_k] \subset V_{j+k}$ for all positive integers $j, k$ ($V_k$ is set to zero if $k > s$); see for instance \cite{MR0367477}.
To see that this is the case, let $\Sigma_1$ be the set of all simple roots (these are positive roots that cannot be decomposed as a sum of two positive roots). It is known that every root $\alpha \in \Sigma^+$ can be written uniquely as a linear combination
$$
\alpha = \sum_{\beta \in \Sigma_1} c_{\beta} \beta,
$$
where each $c_{\beta}$ is a non-negative integer, and not all $c_{\beta}$ in this sum are zero. We then define $d(\alpha)$ to be the positive integer given by the sum of all coefficients $c_{\beta}$ in the above expansion.
For each positive integer $k$, let
$$\Sigma_k := \{ \alpha \in \Sigma^+ \colon d(\alpha) = k\},$$
and let
$$
V_k = \bigoplus_{\alpha \in \Sigma_k} \g_{\alpha}.
$$
It then follows from \eqref{eq:comm_g_alpha} that $[V_j,V_k] \subset V_{j+k}$ for all positive integers $j$ and $k$, and that
$$
\n = \bigoplus_{k=1}^{s} V_k
$$
where $s$ is the largest positive integer for which $V_s \ne \{0\}$.
We are grateful to Roe Goodman for pointing out to us the above argument.

We now introduce exponential coordinates on $N$. Let $\{Y_1,\dots,Y_{m-r}\}$ be a basis of $\n = \bigoplus_{\alpha \in \Sigma^+} \g_{\alpha}$, so that each $Y_j$ belongs to $\g_{\alpha_j}$ for some $\alpha_j \in \Sigma^+$. This is compatible with the grading we had above for $\n$, in the sense that each $Y_j$ belongs to some $V_k$ with $1 \leq k \leq s$: indeed since $\g_{\alpha_j} \subset V_k$ when $k = d(\alpha_j)$, we have $Y_j \in V_{d(\alpha_j)}$ for each $1 \leq j \leq m-r$. Recall now $\M$ is diffeomorphic to $S = AN$, and that $\M$ is simply connected. Thus both $A$ and $N$ are simply connected; as a result, the exponential map $\exp \colon \mathfrak{n} \to N$ is a diffeomorphism of $\n$ onto $N$. Thus we may identify $\exp \left(\sum_{j=1}^{m-r} y^j Y_j \right) \in N$ with $y = (y^1,\dots,y^{m-r}) \in \mathbb{R}^{m-r}$. This is called the exponential coordinates on $N$, associated to the basis $\{Y_1,\dots,Y_{m-r}\}$ of $\n$. We attach a homogeneity $d(\alpha_j)$ to the coordinate $y^j$, for $1 \leq j \leq m-r$. A polynomial in $y$ is then said to be of non-isotropic homogeneous of degree $D$, if it is a linear combination of monomials of the form $y^{j_1} \dots y^{j_{\gamma}}$, with
$$
d(\alpha_{j_1}) + \dots + d(\alpha_{j_{\gamma}}) = D.
$$
Recall that
$$
Y_{j} = \left. \frac{d}{ds} \right|_{s=0} \exp \left(\sum_{\ell=1}^{m-r} y^{\ell} Y_{\ell} \right) \exp(s Y_{j}),
$$
Hence by the Baker-Campbell-Hausdorff formula, there are polynomials $p_{j \ell}(y)$ on $\mathbb{R}^{m-r}$, such that in the coordinates $y$ we had above for $N$, we have
\begin{equation} \label{eq:Ynormalexplicit}
Y_{j} = \sum_{\ell = 1}^{m-r} p_{j {\ell}}(y) \frac{\partial}{\partial y^{\ell}}
\end{equation}
for all $1 \leq j \leq m-r$. Furthermore, each $p_{j {\ell}}$ are non-isotropic homogeneous of degree $d(\alpha_{\ell})-d(\alpha_{j})$ if $d(\alpha_{\ell}) \geq d(\alpha_{j})$, and of degree zero otherwise; in particular,
\begin{equation} \label{eq:Ynormalexplicit2}
\frac{\partial}{\partial y^{\ell}} p_{j \ell}(y) \equiv 0
\end{equation}
for all $1 \leq j, \ell \leq m-r$.
We also have
\begin{equation} \label{eq:nor_center}
p_{j {\ell}}(0) = \delta_{j \ell}.
\end{equation}
See e.g. Rothschild and Stein \cite[Section 10]{MR0436223} for more details.

We remark that if in addition, the basis $\{Y_1, \dots, Y_{m-r}\}$ of $\n$ is orthonormal with respect to the Killing metric $g_0$, then the coefficients of the polynomials $p_{j\ell}(y)$ in \eqref{eq:Ynormalexplicit} are uniformly bounded, with a constant that is
independent of the choice of the basis $\{Y_1, \dots, Y_{m-r}\}$ of $\n$. This is because the iterated commutators of $Y_1, \dots, Y_{m-r}$ have norms (with respect to $g_0$) that are uniformly bounded,
independent of the choice of the basis $\{Y_1, \dots, Y_{m-r}\}$ of $\n$. In particular, this would be the case, if $\{Y_1, \dots, Y_{m-r}\}$ is part of a good basis of $\s = \a \oplus \n$, which we defined in Section~\ref{sect:Kill_metric}.

In a similar manner to the exponential coordinates on $N$, we introduce exponential coordinates on $A$. Let $\{H_1, \dots, H_r\}$ be a basis of $\a$. Then the exponential map $\exp \colon \a \to A$ is a diffeomorphism of $\a$ onto $A$. Thus we may identify $\exp \left( \sum_{i=1}^r t^i H_i \right) \in A$ with $t = (t^1, \dots, t^r) \in \mathbb{R}^r$. This is called exponential coordinates on $A$, associated to the basis $\{H_1, \dots, H_r\}$ of $\a$.

\subsection{Some integration formulae} \label{sect:int}

Let $G = I_0(\M)$, and $K$ be the stabilizer in $G$ of $x_0 \in \M$. Let $G = KAN$ be an Iwasawa decomposition of $G$. We will normalize the Haar measures on $K$, $A$, $N$ and $G$ as follows.

First, let $dk$ be the Haar measure on $K$, normalized so that $\int_K dk = 1$. This is possible since $K$ is compact.

Next, recall the Killing metrics on $A$ and $N$, that we introduced in Section~\ref{sect:Kill_metric}. These left-invariant Riemannian metrics on $A$ and $N$ define Riemannian volume forms, which are Haar measures on $A$ and $N$ respectively. We denote these by $da$ and $dn$ respectively.

More explicitly, let $H_1, \dots, H_r$ be an orthonormal basis of $\a$ with respect to the Killing metric $g_0$. Let $t=(t^1, \dots, t^r)$ be the exponential coordinates on $A$ associated to this basis $H_1, \dots, H_r$, as introduced in the last subsection. Let $da$ then be the Haar measure on $A$, normalized so that $da$ is the Lebesgue measure $dt$ in these exponential coordinates; i.e.
$$\int_A f(a) da = \int_{\mathbb{R}^r} f\left(\exp\left(\sum_{i=1}^r t^i H_i \right) \right) dt,$$
for any compactly supported continuous functions $f$ on $A$.

Similarly, let $Y_1, \dots, Y_{m-r}$ be an orthonormal basis of $\n$ with respect to the Killing metric $g_0$. Let $y=(y^1, \dots, y^{m-r})$ be the exponential coordinates on $N$ associated to this basis $Y_1, \dots, Y_{m-r}$. Let then $dn$ be the Haar measure on $N$, normalized so that $dn$ is the Lebesgue measure $dy$ in these exponential coordinates; i.e.
$$\int_N f(n) dn = \int_{\mathbb{R}^{m-r}} f\left(\exp\left(\sum_{j=1}^{m-r} y^j Y_j \right) \right) dy,$$
for any compactly supported continuous functions $f$ on $N$.

Finally, we define $dg$ to be the Haar measure on $G$, normalized in such a way that the following  two integration formula
of Harish-Chandra \cite[Lemma 35]{MR0056610} hold: for any compactly supported continuous function $F$ on $G$, we have
\begin{equation} \label{eq:d(ank)}
\int_G F(g) dg = \int_K \int_A \int_N F(ank) \,dn\,da\,dk,
\end{equation}
and
\begin{equation} \label{eq:d(kan)}
\int_G F(g) dg = \int_K \int_A \int_N F(kan) e^{2\rho(\log a)} \,dn\,da\,dk.
\end{equation}
Here
$\rho$ is the half sum of positive roots given by \eqref{eq:rhodef}, and
for each $a \in A$, $\log a$ is the element in the Lie algebra $\a$ of $A$ for which $e^{\log a} = a$.
Indeed, \eqref{eq:d(kan)} can be deduced from \eqref{eq:d(ank)} as follows: if $F$ is a compactly supported continuous function on $G$, then
$$
\int_G F(g) dg = \int_G F(g^{-1}) dg = \int_K \int_A \int_N F(k^{-1} n^{-1} a^{-1}) \,dn\,da\,dk
$$
where in the last equality, we have applied \eqref{eq:d(ank)} to the function $g \mapsto F(g^{-1})$. This shows
\begin{equation} \label{eq:d(kna)}
\int_G F(g) dg = \int_K \int_A \int_N F(kna) \,dn\,da\,dk,
\end{equation}
and upon the change of variables $n \mapsto a n a^{-1}$, we get
\begin{equation} \label{eq:d(kanAd)}
\int_G F(g) dg = \int_K \int_A \int_N F(kan) \det(\left. (\Ad a) \right|_{\n}) \,dn\,da\,dk.
\end{equation}
Here for each $a \in A$, we think of $\Ad a$ as a linear map in $GL(\g)$, which restricts to a map
\begin{align*}
\left. \Ad a \right|_{\n} & \colon \n \to \n \\
\left. \Ad a \right|_{\n}(Y) &= \left. \frac{d}{ds} \right|_{s=0} a e^{sY} a^{-1}.
\end{align*}
Note also that we could also have written
\begin{equation} \label{eq:e2rho}
e^{2\rho(\log a)} = \det(\left. (\Ad a) \right|_{\n});
\end{equation}
this holds because $\Ad a = e^{\ad (\log a)}$, from which it follows that $$\det(\left. (\Ad a) \right|_{\n}) = e^{\operatorname{trace} (\left. \ad (\log a) \right|_{\n} )} = e^{2\rho(\log a)}.$$
\eqref{eq:d(kanAd)} and \eqref{eq:e2rho} together gives \eqref{eq:d(kan)}.

We note here that the definition of $dg$ does not depend on the choice of the Iwasawa decomposition $KAN$ of $G$. Indeed, if $G = KA_0N_0$ is another Iwasawa decomposition of $G$, then there exists $k \in K$ such that $A = kA_0k^{-1}$, and $N=kN_0k^{-1}$; see for instance Remark after Theorem 4.6 of \cite{MR1476489}. If $da_0$ and $dn_0$ are the normalized Haar measures on $A_0$ and $N_0$ respectively, then $da$ is the pull-back of $da_0$ under the map $\Ad(k) \colon A \to A_0$, and $dn$ is the pull-back of $dn_0$ under the map $\Ad(k) \colon N \to N_0$. Thus
$$
\int_K \int_A \int_N F(ank) \,dn\,da\,dk = \int_K \int_{A_0} \int_{N_0} F(a_0 n_0 k) \,dn_0\, da_0 \,dk
$$
for all continuous functions $F$ with compact support in $G$, and this shows that $dg$ is well-defined independent of the choice of the Iwasawa decomposition of $G$.

Now consider the Lie subgroup $S=AN=NA$ of $G$. We define a Haar measure $ds$ on $S$, by
$$
\int_S F(s) ds = \int_A \int_N F(an) dnda
$$
for all compactly supported continuous functions $F$ on $S$.
To see that this indeed defines a Haar measure on $S$, it suffices to show that the linear functional on the right hand side above is invariant under both left translations by $A$ and $N$. But for any $a_0 \in A$, we have
$$
\int_A \int_N F(a_0an) \,dn\,da = \int_A \int_N F(an) \,dn\,da,
$$
since $da$ is a Haar measure on $A$. Also, for any $n_0 \in N$, we have
$$
\int_A \int_N F(n_0an) \,dn\,da = \int_A \int_N F(a (a^{-1} n_0 a) n) \,dn\,da = \int_A \int_N F(an) \,dn\,da,
$$
the last equality following from that $a^{-1} n_0 a \in N$, and from that $dn$ is a Haar measure on $N$. This shows that $ds$ is indeed a Haar measure on $S$, and we call it the normalized Haar measure on $S$ (since $da$ and $dn$ are normalized by our convention).

The Lie subgroup $S$ of $G$ can be identified with our symmetric space $\M$, as we have seen at the beginning of Section~\ref{sect:Kill_metric}. Recall that we had a Riemannian metric $g$ on $\M$ in the statement of Theorem~\ref{thm:main}, and that we are assuming that $g = g_0$ the Killing metric. There all integration are with respect to the Riemannian volume element $dV$. But $dV$ can be identified, through the identification of $\M$ with $S$, with a Haar measure on $S$. Since the space of Haar measure on $S$ is one-dimensional, this shows that there is a non-zero constant $c_0$, for which $dV = c_0 ds$. By examining the two sides near $x_0$, we see that $c_0 = 1$. In other words,
if $F_0$ is a continuous function with compact support on $\M$, then
\begin{equation} \label{eq:intAN=intM}
\int_{\M} F_0 \, dV
= 
\int_A \int_N F_0(anx_0) \,dn\,da.
\end{equation}
Note that \eqref{eq:intAN=intM} can also be written as
\begin{equation} \label{eq:intNA=intM}
\int_{\M} F_0 \, dV
= 
\int_A \int_N F_0(nax_0) e^{-2\rho(\log a)} \,dn\,da.
\end{equation}
Indeed, this follows from \eqref{eq:intAN=intM} upon the change of variable $n \mapsto a^{-1} n a$, using \eqref{eq:e2rho}.

Furthermore, we claim that if $F_0$ is a continuous function with compact support on $\M$, then
\begin{equation} \label{eq:dgK=dV}
\int_{\M} F_0 \, dV
= 
\int_G F_0(gx_0) dg.
\end{equation}
Indeed it suffices to compute the right hand side using \eqref{eq:d(ank)}, and to note that $kx_0 = x_0$ for all $k \in K$, while $\int_K dk = 1$. \eqref{eq:dgK=dV} then follows from \eqref{eq:intAN=intM}.

Suppose now $H_1$ is a non-zero vector in $\a$. Let $\a'$ denote the orthogonal complement to $H_1$ in $\a$ with respect to the Killing metric $g_0$, and let $A'$ be the Lie subgroup of $A$ generated by $\a'$. Let $\s' = \a' \oplus \n$, and $S'=A'N =NA'$. Then completely analogous to what we have done above, we can introduce normalized Haar measures $da'$ and $ds'$ on $A'$ and $S'$ respectively; indeed $da'$ will be the volume element, given by the left-invariant Riemannian metric on $A'$, that one obtains by restricting the Killing metric $g_0$ to $\a'$, and $ds'$ will be given by
$$
\int_{S'} F(s') ds' = \int_{A'} \int_N F(a'n) da' dn.
$$
We can also introduce exponential coordinates on $A'$ as in the last subsection. If $\{H_2, \dots, H_r\}$ is a basis of $\a'$, then the exponential map $t'=(t^2, \dots, t^r) \mapsto \exp(\sum_{i=2}^r t^i H_i)$ defines a diffeomorphism of $\mathbb{R}^{r-1}$ onto $A'$. If further $\{H_2,\dots,H_r\}$ is orthonormal with respect to the Killing metric $g_0$, then in such exponential coordinates, $da'$ is just the Lebesgue measure $dt'$.

Finally, we need two well-known integration by parts lemma. Suppose first $A'$ and $da'$ are defined as above. Let $\{H_2, \dots, H_r\}$ be a basis of left-invariant vector fields on $A'$.
\begin{lem} \label{lem:intHi}
For any function $\varphi \in C^{\infty}_c(A')$, we have
$$
\int_{A'} H_i \varphi \, da' = 0
$$
for $2 \leq i \leq r$.
\end{lem}

\begin{proof}
Indeed, for each $2 \leq i \leq r$, $H_i$ is just the coordinate derivative $\frac{\partial}{\partial t^i}$ in the exponential coordinates associated to $\{H_2,\dots,H_r\}$. Since $da' = dt'$ in such exponential coordinates, the lemma follows.
\end{proof}

Next, pick a basis $\{Y_1, \dots, Y_{m-r}\}$ of left-invariant vector fields on $N$.

\begin{lem} \label{lem:intYi}
For any function $\varphi \in C^{\infty}_c(N)$, we have
$$
\int_N Y_j \varphi \, dn = 0
$$
for $1 \leq j \leq m-r$.
\end{lem}

\begin{proof}
One just writes everything out in exponential coordinates. Without loss of generality, assume each $Y_{j} \in \g_{\alpha_j}$ for some $\alpha_j \in \Sigma^+$. Then $Y_{j}$ takes the form \eqref{eq:Ynormalexplicit}, where the coefficients $p_{j \ell}$ satisfies \eqref{eq:Ynormalexplicit2}, and the Haar measure $dn$ is just a scalar multiple of the Lebesgue measure in the exponential coordinates associated to $\{Y_1, \dots, Y_{m-r}\}$. So one can finish the proof of the lemma by using ordinary integration by parts on $\mathbb{R}^{m-r}$.
This gives the assertion of Lemma~\ref{lem:intYi}.
\end{proof}

We note that the validity of the above two lemmas can partly be explained by the fact that the geodesics $s \mapsto \exp(sY_j)$ and $s \mapsto \exp(sH_i)$ go off to infinity; in particular, they cannot be closed curves, cf. Appendix B.13 in Ballmann \cite{MR2243012}.

\section{A decomposition lemma} \label{sect:decomp}

Let $G = I_0(\M)$ be the identity component of the isometry group of $\M$. Fix a point $x_0 \in \M$, and let $K \subset G$ be the stabilizer of $x_0$. Let $G = KAN$ be a corresponding Iwasawa decomposition of $G$, and $\k$, $\a$, $\n$ be the Lie algebras of $K$, $A$, $N$ respectively.

Suppose $H_1 \in \a$ is a unit vector with respect to the Killing metric $g_0$ introduced in Section~\ref{sect:Kill_metric}. Let $\a'$ be the orthogonal complement of $H_1$ in $\a$ under $g_0$. Let $A'$ be the Lie subgroup of $A$ with Lie algebra $\a'$, and let $S'=A'N=NA'$. We will need a decomposition lemma for functions $\Phi$ defined on $S'$. The constants that appear in the lemma has to be independent of the choice of the Iwasawa decomposition $KAN$ of $G$, and independent of the choice $H_1 \in \a$. We will now proceed to formulate this lemma.

Let $da'$, $dn$ and $ds'$ be the normalized Haar measures on $A'$, $N$ and $S'$ as defined in Section~\ref{sect:int}. In particular, $da'$ and $dn$ are the Riemannian volume elements associated to the Killing metric on $A'$ and $N$, and
$$
\int_{S'} \Phi(s') ds' = \int_{A'} \int_N \Phi(a'n) da' dn
$$
for any compactly supported continuous functions $\Phi$ on $S'$. For any such $\Phi$, and any $1 \leq p < \infty$, we also denote the $L^p(S')$ norm of $\Phi$ by
$$
\Vert \Phi\Vert_{L^p(S')} = \left( \int_{S'} |\Phi(s')|^p ds' \right)^{1/p},
$$
and the $L^{\infty}(S')$ norm of $\Phi$ by $\Vert \Phi\Vert_{L^{\infty}(S')} = \sup_{s' \in S'} |\Phi(s')|$. For any smooth functions $\Phi$ with compact support on $S'$, we also write
$$
|\nabla' \Phi(s')| = \left( \sum_{\ell = 2}^m |X_{\ell} \Phi(s')|^2 \right)^{1/2}
$$
where $\{X_2, \dots, X_m\}$ is an orthonormal frame of tangent vectors with respect to the Killing metric on $S'$. Finally, for $1 \leq p \leq \infty$, we write $\Vert \nabla' \Phi\Vert_{L^p(S')}$ for the $L^p(S')$ norm of $|\nabla' \Phi|$, and we write
$$
\Vert  \Phi\Vert_{W^{1,p}(S')}
:= \Vert \Phi\Vert_{L^p(S')} + \Vert \nabla' \Phi\Vert_{L^p(S')}.
$$

We can now formulate our decomposition lemma.

\begin{lem} \label{lem:decompS'unif}
There exists a constant $C > 0$ such that the following holds. Let $\Phi$ be a smooth function with compact support on the  Lie group $S'$. Suppose $p > \textrm{dim}(S') = m-1$, and $\lambda > 0$. Then there exists a decomposition $\Phi = \Phi_1 + \Phi_2$ on $S'$, such that $\Phi_1, \Phi_2 \in C^{\infty}_c(S')$, and
\begin{equation} \label{eq:decomp_conclu}
\begin{cases}
\Vert \Phi_1\Vert_{L^{\infty}(S')} &\leq C \lambda^{1-\frac{m-1}{p}} \Vert  \Phi\Vert_{W^{1,p}(S')}  \\
\Vert \Phi_2\Vert_{L^{\infty}(S')} &\leq C \lambda^{-\frac{m-1}{p}} \Vert \Phi\Vert_{L^p(S')} \\
\Vert \nabla^{'} \Phi_2\Vert_{L^{\infty}(S')} &\leq C \lambda^{-\frac{m-1}{p}} \Vert  \Phi\Vert_{W^{1,p}(S')}.
\end{cases}
\end{equation}
The constant $C$ will be chosen independent of the choice of the Iwasawa decomposition $KAN$ of $G$, and independent of the choice of $H_1 \in \a$.
\end{lem}

To prove Lemma~\ref{lem:decompS'unif}, we need the following decomposition lemma on $\mathbb{R}^{d}$, with $d = m-1$:

\begin{lem} \label{lem:decompRd}
Let $\Phi \in C^{\infty}_c(\mathbb{R}^d)$. Suppose $p > d$, and $\lambda > 0$. Then there exists a decomposition $\Phi = \Phi_1 + \Phi_2$ on $\mathbb{R}^d$, such that $\Phi_1, \Phi_2 \in C^{\infty}_c(\mathbb{R}^d)$, and
\begin{equation} \label{eq:decomp_conclu_Rd}
\left\{
\begin{aligned}
&\Vert \Phi_1\Vert_{L^{\infty}(\mathbb{R}^d)} &\leq C \lambda^{1-\frac{d}{p}} \Vert \partial_x \Phi\Vert_{L^p(\mathbb{R}^d)}  \\
&\Vert \Phi_2\Vert_{L^{\infty}(\mathbb{R}^d)} &\leq C \lambda^{-\frac{d}{p}} \Vert \Phi\Vert_{L^p(\mathbb{R}^d)} \\
&\Vert \partial_x \Phi_2\Vert_{L^{\infty}(\mathbb{R}^d)} &\leq C \lambda^{-\frac{d}{p}} \Vert \partial_x \Phi\Vert_{L^p(\mathbb{R}^d)}.
\end{aligned}
\right.
\end{equation}
Furthermore, if $\Phi$ is supported on a ball of radius $R$, then one can arrange $\Phi_1$ and $\Phi_2$, so that they are both supported on a ball of radius $R + \lambda$.
\end{lem}

The proof is implicit already in the proof of Theorem 1.5 in \cite{MR2078071}; see also \cite{MR2122730}. We briefly outline it for the convenience of the reader.

\begin{proof}
Since $\partial_x \Phi \in L^p(\mathbb{R}^d)$  with $p > d$, Morrey's embedding implies that $\Phi \in C^{\gamma}$ where $\gamma = 1-\frac{d}{p}$. It is then well-known that $\Phi$ can be decomposed into $\Phi = \Phi_1 + \Phi_2$, as in the statement of the lemma: indeed one just needs to take
$$
\Phi_2 := \Phi * \eta_{\lambda}, \quad \Phi_1 := \Phi - \Phi_2
$$
where $\eta$ is a fixed smooth function with compact support on the unit ball on $\mathbb{R}^d$ with $\int_{\mathbb{R}^d} \eta = 1$, and $\eta_{\lambda}(x) := \lambda^{-d} \eta (\lambda^{-1} x)$. Morally speaking, $\Phi_2$ is the low-frequency component of $\Phi$, and $\Phi_1$ the high-frequency component. It is then straightforward to establish the estimates in \eqref{eq:decomp_conclu_Rd}, and the estimates on the sizes of the supports of $\Phi_1$ and $\Phi_2$; see for instance, Corollary 1 in \cite[Chapter VI, Section 5.3]{MR1232192}, for a variant of this argument.
\end{proof}

We are now ready for the proof of Lemma~\ref{lem:decompS'unif}.

\begin{proof}[Proof of Lemma~\ref{lem:decompS'unif}]
First we introduce, around any point $s'_0 \in S'$, exponential coordinates $x':=(t',y)$ as follows. Recall that we have picked a vector $H_1 \in \a$, that has unit length with respect to the Killing metric $g_0$. We now complete it to a good basis of $\s = \a \oplus \n$, as defined in Section~\ref{sect:Kill_metric}; in other words, we pick $H_2, \dots, H_r \in \a$, and $Y_1, \dots, Y_{m-r} \in \n$, such that each $Y_j$ belongs to $\g_{\alpha_j}$ for some $\alpha_j \in \Sigma^+$, and such that $\{H_1, \dots, H_r, Y_1, \dots, Y_{m-r}\}$ form an orthonormal basis of $\s$ with respect to $g_0$. This is possible by the orthogonality relations \eqref{item:ortho1} and \eqref{item:ortho2} in Section~\ref{sect:Kill_metric}. Later on we will want our constants to be chosen independent of the choice of the Iwasawa decomposition $KAN$ of $G$, and independent of the choice of good basis $\{H_1, \dots, H_r, Y_1, \dots, Y_{m-r}\}$ of $\s$. For brevity, such constants will be called \emph{absolute constants}.

Now suppose we fix a point $s'_0 \in S'$. We identify $x':=(t',y) \in \mathbb{R}^{r-1} \times \mathbb{R}^{m-r}$ with the point
$$
s'_0 \exp \left( \sum_{j=1}^{m-r} y^j Y_j \right) \exp \left( \sum_{i=2}^r t^i H_i \right).
$$
Then in these exponential coordinates on $S'$, $H_i$ is just
\begin{equation} \label{eq:Hnormalexplicit}
H_i = \frac{\partial}{\partial t^i}
\end{equation}
for $2 \leq i \leq r$. We claim that $Y_j$ is given, in these coordinates, by
\begin{equation} \label{eq:Ynormalexplicit3}
Y_j =  \left( \prod_{i=2}^r e^{t^i \alpha_j(H_i)} \right) \sum_{\ell=1}^{m-r} p_{j \ell}(y) \frac{\partial}{\partial y^{\ell}}.
\end{equation}
where $p_{j \ell}(y)$ are defined as in \eqref{eq:Ynormalexplicit}.
Indeed,
$$
Y_j = \left. \frac{d}{ds} \right|_{s=0} s'_0 \exp \left( \sum_{\ell=1}^{m-r} y^{\ell} Y_{\ell} \right) \exp \left( \sum_{i=2}^r t^i H_i \right) \exp(s Y_j),
$$
whereas we will rewrite the last two factors as
$$
\exp \left( \sum_{i=2}^r t^i H_i \right) \exp(s Y_j) \exp \left( -\sum_{i=2}^r t^i H_i \right) \exp \left( \sum_{i=2}^r t^i H_i \right).
$$
Now
\begin{align*}
&\exp \left( \sum_{i=2}^r t^i H_i \right) \exp(s Y_j) \exp \left( -\sum_{i=2}^r t^i H_i \right)  \\
=&  \exp \left[ \left( \Ad \exp \left( \sum_{i=2}^r t^i H_i \right) \right) s Y_j \right]\\\
=& \exp \left[ \exp \left( \ad \left( \sum_{i=2}^r t^i H_i \right) \right) s Y_j \right] \\
=& \exp \left[ s \prod_{i=2}^r e^{t^i \alpha_j(H_i) } Y_j  \right]
\end{align*}
Putting these back, we have
\begin{align*}
Y_j &= \left. \frac{d}{ds} \right|_{s=0} s'_0 \exp \left( \sum_{\ell=1}^{m-r} y^{\ell} Y_{\ell} \right) \exp \left[ s \prod_{i=2}^r e^{t^i \alpha_j(H_i) } Y_j  \right]  \exp \left( \sum_{i=2}^r t^i H_i \right) \\
&= \prod_{i=2}^r e^{t^i \alpha_j(H_i) } \left. \frac{d}{ds} \right|_{s=0} s'_0 \exp \left( \sum_{\ell=1}^{m-r} y^{\ell} Y_{\ell} \right) \exp (s Y_j ) \exp \left( \sum_{i=2}^r t^i H_i \right)
\end{align*}
Using the Baker-Campbell-Hausdorff formula as we used to prove \eqref{eq:Ynormalexplicit}, we obtain \eqref{eq:Ynormalexplicit3}.

In view of \eqref{eq:Hnormalexplicit}, \eqref{eq:Ynormalexplicit3} and \eqref{eq:nor_center}, we see that if $s'_0$ is any point on $S'$ and $(t',y)$ is the exponential coordinates we introduced above, then at $(t',y) = (0,0)$, we have
\begin{equation} \label{eq:normalcenter}
\left. H_i \right|_{s'_0} = \left. \frac{\partial}{\partial t^i} \right|_{(0,0)}, \quad \left. Y_j \right|_{s'_0} = \left. \frac{\partial}{\partial y^j} \right|_{(0,0)}.
\end{equation}
Now write $B(s'_0,r)$ for the metric ball on $S'$, with respect to the Killing metric on $S'$, that is centered at $s'_0$ and of radius $r$. Then in view of the above, there exists $\varepsilon_0 > 0$ such that if $s' \in B(s'_0,\varepsilon_0)$ and $x'=(t',y)$ is the exponential coordinates of $s'$ centered at $s'_0$, then for any function $\Phi$ on $S'$, we have
$$
|\nabla^{'} \Phi(s')| \simeq
\sum_{i=2}^r \left| \frac{\partial}{\partial t^i} \Phi(t',y) \right| + \sum_{j=1}^{m-r} \left| \frac{\partial}{\partial y^j} \Phi(t',y) \right|,
$$
which we abbreviate as
\begin{equation} \label{eq:grad_comp_smallball}
|\nabla^{'} \Phi(s')| \simeq
\left| \frac{\partial \Phi}{\partial x'}(x') \right|.
\end{equation}
Indeed $\varepsilon_0$ and the implicit constants here can be chosen independent of $s'_0$, and are hence uniform over $S'$; they can also be chosen to be independent of the choice of $\{H_1,\dots,H_r\}$, since each $\alpha_j$ remains uniformly bounded on the unit sphere in $\a$ (where we put the Killing metric on $\a$). Furthermore, $\varepsilon_0$ and the implicit constants above can be chosen to be independent of the choice of $\{Y_1, \dots, Y_{m-r}\}$, since the coefficients of the polynomials $p_{j\ell}(y)$ in \eqref{eq:Ynormalexplicit} are uniformly bounded, independent of the choice of $\{Y_1, \dots, Y_{m-r}\}$; see discussion near the end of Section~\ref{sect:expcoords}. We further claim that $\varepsilon_0$ and the implicit constants above can be chosen to be independent of the choice of the Iwasawa decomposition $KAN$ of $G$: indeed, if $G = KA_0N_0$ is another Iwasawa decomposition of $G$, then there exists $k \in K$ such that $A = kA_0k^{-1}$, and $N=kN_0k^{-1}$; see for instance Remark after Theorem 4.6 of \cite{MR1476489}. Since for $k \in K$, the maps $\Ad(k) \colon A \to A_0$ and $\Ad(k) \colon N \to N_0$ are isometries with respect to the Killing metric, our claim follows. It follows that $\varepsilon_0$ and the implicit constants in \eqref{eq:grad_comp_smallball} are absolute constants.

Since the measure $ds'$ is a normalized Haar measure, we also have the following uniform comparison of $ds'$ with the Lebesgue measure $dx'$ on $\mathbb{R}^{m-1}$: indeed
\begin{equation} \label{eq:meas_comp_ball}
\int_{B(s'_0,\varepsilon_0)} |\Phi(s')| ds' \simeq \int_{B(s'_0,\varepsilon_0)} |\Phi(x')| dx'
\end{equation}
holds uniformly for any function $\Phi$ on $S'$ and any $s'_0 \in S'$; here the implicit constants are again absolute constants.

Now to prove Lemma~\ref{lem:decompS'unif}, suppose $\Phi \in C^{\infty}_c(S')$, $p > \text{dim}(S') = m-1$, and $\lambda > 0$ are given. We consider two cases.

\vspace{0.1 in}

\noindent{\textbf{Case 1}: $\lambda \geq C_1^{-1} \varepsilon_0$, where $C_1$ is a large absolute constant to be chosen.}

\vspace{0.1 in}

Then take $\Phi_1 = \Phi$, $\Phi_2 = 0$. We claim that
$$
\Vert \Phi\Vert_{L^{\infty}(S')} \leq C \Vert \Phi\Vert_{W^{1,p}(S')},
$$
where $C$ is an absolute constant.
In fact, given $s'_0 \in S'$, the exponential coordinates centered at $s'_0$ allows one to identify $B(s'_0,\varepsilon_0)$ with a bounded open set in $\mathbb{R}^{m-1}$. Then by the Sobolev inequality on $\mathbb{R}^{m-1}$, we have
$$
\Vert \Phi\Vert_{L^{\infty}(B(s'_0,\varepsilon_0))} \leq C \left\Vert  \, |\Phi| + \left| \frac{\partial \Phi}{\partial x'} \right| \, \right\Vert_{L^p(B(s'_0,\varepsilon_0),dx')}.
$$
(Here we use $p > m-1$.) But by \eqref{eq:grad_comp_smallball} and \eqref{eq:meas_comp_ball}, the right hand side is bounded by a constant multiple of $\Vert \Phi\Vert_{W^{1,p}(S')}$. All constants here are absolute constants.
This implies our desired claim, and completes our proof in Case 1.

\vspace{0.1 in}

\noindent{\textbf{Case 2}: $\lambda < C_1^{-1} \varepsilon_0$.}

\vspace{0.1 in}

We choose a suitable smooth function $\chi$ on $S'$, supported in a ball of sufficiently small radius centered at the identity element of $S'$, such that
$$
1 = \int_{s_0' \in S'} \chi((s'_0)^{-1} s') ds'_0  \quad \text{for all $s' \in S'$}.
$$
(Here $ds_0'$ is the normalized Haar measure on $S'$, and hence this integral is constant on $S'$.) Then given $\Phi \in C^{\infty}_c(S')$, we write
$$
\Phi(s') = \int_{S'} \chi((s'_0)^{-1} s') \Phi(s') ds'_0.
$$
Now since $\lambda \geq C_1^{-1} \varepsilon_0$, then we decompose, for each $s'_0 \in S'$, the function $$\Phi^{(s'_0)}(s'):= \chi((s'_0)^{-1} s') \Phi(s')$$ by first identifying it with a function on $\mathbb{R}^{m-1}$, and then applying the decomposition lemma on $\mathbb{R}^{m-1}$. Using (\ref{eq:grad_comp_smallball}) and (\ref{eq:meas_comp_ball}), one then obtains, for each $s'_0 \in S'$, a decomposition
$$
\Phi^{(s'_0)} = \Phi^{(s'_0)}_1 + \Phi^{(s'_0)}_2,
$$
where $\Phi^{(s'_0)}_1$ and $\Phi^{(s'_0)}_2$ satisfies
\begin{align*}
\Phi^{(s'_0)}_1(s') &\leq C \lambda^{1-\frac{m-1}{p}}  \left\Vert \nabla' \Phi^{(s'_0)} \right\Vert_{L^p(S')} \chi_{B(s_0',\varepsilon_0)}(s') \\
\Phi^{(s'_0)}_2(s') &\leq C \lambda^{-\frac{m-1}{p}}  \Vert \Phi^{(s'_0)}\Vert_{L^p(S')} \chi_{B(s_0',\varepsilon_0)}(s') \\
|\nabla' \Phi^{(s'_0)}_2(s')| & \leq C  \lambda^{-\frac{m-1}{p}} \left\Vert \nabla' \Phi^{(s'_0)} \right\Vert_{L^p(S')} \chi_{B(s_0',\varepsilon_0)}(s').
\end{align*}
(The characteristic functions on the right hand side indicate that $\Phi^{(s'_0)}_1$ and $\Phi^{(s'_0)}_2$ are both supported in a ball centered at $s'_0$ and of radius $\varepsilon_0$.)
We now integrate both sides of each inequality above over $s'_0 \in S'$ with respect to $ds'_0$. Then defining
$$
\Phi_1 := \int_{s'_0 \in S'} \Phi^{(s'_0)}_1 ds'_0, \quad \Phi_2 := \int_{s'_0 \in S'} \Phi^{(s'_0)}_2 ds'_0,
$$
we get \eqref{eq:decomp_conclu} as desired, with constants that are absolute constants.
\end{proof}

\section{Proof of Theorem~\ref{thm:main}} \label{sect:proof}

We now prove Theorem~\ref{thm:main}. As explained at the end of Section~\ref{sect:Kill_metric}, we may take the Riemannian metric $g$ on $\M$ to be the Killing metric $g_0$. Henceforth the pointwise inner product $\langle \cdot, \cdot \rangle$, the pointwise norm $| \cdot |$, the volume form $dV$ and the Levi-Civita connection $\nabla$ will all be taken with respect to the Killing metric $g_0$.

Recall that we fixed a point $x_0 \in \M$. Let $T_{x_0} \M$ be the tangent space to $\M$ at $x_0$. Let $\mathbb{S}^{m-1} \subset T_{x_0}\M$ be the unit sphere in $T_{x_0} \M$ under the inner product $\langle \cdot, \cdot \rangle$ induced by $g_0$. To prove Theorem~\ref{thm:main}, we need to estimate the integral
$$
\int_{\M} \langle f, \phi \rangle dV.
$$
We will do so by rewriting the integral as an integral over $G \times \mathbb{S}^{m-1}$, where $G := I_0(\M)$ is the identity component of the isometry group of $\M$. Indeed, let $L_g \colon \M \to \M$ be the left translation by $g \in G$, i.e. $L_g x = gx$ for $g \in G$ and $x \in \M$. Then if $v$ is a tangent vector to $\M$ at some point $x \in \M$, $dL_g(v)$ is a tangent vector to $\M$ at $gx$, and we abbreviate this by
$$
d_g v := dL_g(v).
$$
Now fix a unit vector $v_0 \in T_{x_0} \M$. Then $d_g v_0$ will be a unit tangent vector to $\M$ at $gx_0$. We claim that there exists a non-zero constant $C_0$, such that
\begin{equation} \label{eq:Mav=KOav}
C_0 \int_{\M} \langle f, \phi \rangle dV = \int_G \int_{\mathbb{S}^{m-1}} \langle f(gx_0), d_g v_0 \rangle \langle \phi(gx_0), d_g v_0 \rangle \, d\sigma(v_0) \, dg.
\end{equation}
Here $dg$ is the normalized Haar measure on $G$ as in Section~\ref{sect:int}, and $d\sigma(v_0)$ is the standard surface measure on $\mathbb{S}^{m-1}$, normalized such that $\int_{\mathbb{S}^{m-1}} d\sigma(v_0) = 1$.

To verify \eqref{eq:Mav=KOav}, fix $g \in G$. The inner integrand on the right hand side of \eqref{eq:Mav=KOav} is just
$$
\int_{\mathbb{S}^{m-1}} \langle d_{g^{-1}} f(gx_0), v_0 \rangle \langle d_{g^{-1}} \phi(gx_0), v_0 \rangle \, d\sigma(v_0).
$$
But $d_{g^{-1}} f(gx_0)$ and $d_{g^{-1}} \phi(gx_0)$ are just two tangent vectors to $\M$ at $x_0$. One can show that for any tangent vectors $u, u' \in T_{x_0}\M$, we have
$$
\int_{\mathbb{S}^{m-1}} \langle u, v_0 \rangle \langle u',v_0 \rangle \, d\sigma(v_0) = C_0 \langle u, u' \rangle
$$
where $C_0$ is a non-zero constant; this just follows by a direct calculation on $\mathbb{R}^m$. Thus the inner integral on the right hand side of \eqref{eq:Mav=KOav} is just
$$
C_0 \langle d_{g^{-1}} f(gx_0), d_{g^{-1}} \phi(gx_0) \rangle = C_0 \langle f(gx_0), \phi(gx_0) \rangle.
$$
By \eqref{eq:dgK=dV}, it then follows that the right hand side of \eqref{eq:Mav=KOav} is equal to
$$
\int_{G} C_0 \langle f(gx_0), \phi(gx_0) \rangle dg = C_0 \int_{\M} \langle f, \phi \rangle dV.
$$
So \eqref{eq:Mav=KOav} follows.

Now to prove Theorem~\ref{thm:main}, we just need to estimate the right hand side of \eqref{eq:Mav=KOav}. Indeed, we will reverse the order of integration, and estimate the integral over $G$ for each fixed $v_0 \in \mathbb{S}^{m-1}$. More precisely, we will establish the following proposition:

\begin{prop} \label{prop:G_int}
Suppose $f$ is a smooth vector field on $\M$ with $\div f = 0$, and $\phi$ is a smooth vector field on $\M$ with compact support. Then for each vector $v_0 \in \mathbb{S}^{m-1} \subset T_{x_0}\M$, we have
\begin{equation} \label{eq:intG_est}
\left| \int_G \langle f(gx_0), d_g v_0 \rangle \langle \phi(gx_0), d_g v_0 \rangle dg \right| \leq C \Vert f\Vert_{L^1(dV)} \Vert \nabla \phi\Vert_{L^m(dV)},
\end{equation}
where the constant $C$ is independent of $v_0 \in \mathbb{S}^{m-1}$.
\end{prop}

If this is established, then integrating \eqref{eq:intG_est} over $v_0 \in \mathbb{S}^{m-1}$, we obtain the assertion of Theorem~\ref{thm:main}.

We remark that when $\M$ has rank 1, it is easy to show that the constant $C$ in Proposition~\ref{prop:G_int} is independent of $v_0$, since $G$ acts transitively on $\mathbb{S}^{m-1}$. This is no longer true in the higher rank case. So we are forced to use a family of Iwasawa decompositions in what follows, and obtain uniform estimates for such.

To establish Proposition~\ref{prop:G_int}, let $v_0 \in \mathbb{S}^{m-1}$. We use a suitable Iwasawa decomposition of $G$ adapted to $v_0$. Let $\g$ be the Lie algebra of $G$, let $K \subset G$ be the stabilizer of $x_0$, and let $\k$ be the Lie algebra of $K$. We write $$\g = \k \oplus \p$$ for the Cartan decomposition of $\g$. Given $v_0 \in \mathbb{S}^{m-1}$, we select a vector $H_1 \in \p$, that has unit length with respect to $g_0$, such that
\begin{equation} \label{eq:H1def}
\left. \frac{d}{dt} \right|_{t=0} e^{t {H_1}} x_0 = v_0.
\end{equation}
Then we choose a maximal abelian subalgebra $\a$ of $\g$ contained in $\p$, such that $H_1 \in \a$.
Let $\a^*$ be the set of all linear functionals on $\a$, and $\Sigma$ be the set of restricted roots of $\a$.
We make a choice $\Sigma^+$ of positive roots in $\Sigma$.
Writing
$$
\n := \bigoplus_{\alpha \in \Sigma^+} \g_{\alpha},
$$
and $A$, $N$ for the connected Lie groups with Lie algebras $\a$ and $\n$ respectively, we then have an Iwasawa decomposition
$
G = KAN 
$
of $G$. We determine normalized Haar measures $dk$, $da$, $dn$ and $dg$ on $K$, $A$, $N$ and $G$ respectively, as in Section~\ref{sect:int}.

Now the integral in left hand side of \eqref{eq:intG_est} can be written, via \eqref{eq:d(kan)}, as
\begin{equation} \label{eq:intKAN}
\int_K \int_A \int_N \langle f(kanx_0), d_{kan} v_0 \rangle \langle \phi(kanx_0), d_{kan} v_0 \rangle \,dn\,da\,dk.
\end{equation}
To estimate this, recall we had chosen a unit vector $H_1 \in \a$ such that  \eqref{eq:H1def} holds. We complete $H_1$ to a good basis on $\s = \a \oplus \n$; in other words, we pick $H_2, \dots, H_r \in \a$, and $Y_1, \dots, Y_{m-r} \in \n$, such that $\{H_1, \dots, H_r, Y_1, \dots, Y_{m-r}\}$ form an orthonormal basis of $\s$ with respect to $g_0$, and such that each $Y_j$ belongs to $\g_{\alpha_j}$ for some $\alpha_j \in \Sigma^+$. This is possible by the orthogonality relations \eqref{item:ortho1} and \eqref{item:ortho2} in Section~\ref{sect:Kill_metric}.

Now let $\a_1$, $\a'$ be the spans of $\{H_1\}$ and $\{H_2,\dots,H_r\}$ respectively. Let $A_1$ and $A'$ be the Lie subgroups of $A$ with Lie algebras $\a_1$ and $\a'$ respectively. Let $S'=A'N=NA'$. Let $da_1$, $da'$ be Haar measures on $A_1$ and $A'$, given by the Riemannian volume elements associated to the restriction of $g_0$ to $A_1$ and $A'$. Then
$$
\int_A F(a) da = \int_{A_1} \int_{A'} F(a_1 a') da_1 da'
$$
for all compactly supported continuous functions $F$ on $A$. We will use this to evaluate the integral over $A$ in \eqref{eq:intKAN}. The key then is the following basic estimate:
\begin{prop} \label{prop:basicest}
Suppose $f$ is a smooth vector field on $\M$ with $\div f = 0$, and $\phi$ is a smooth vector field on $\M$ with compact support. For any vector $v_0 \in \mathbb{S}^{m-1} \subset T_{x_0}\M$, let $KAN$ be an Iwasawa decomposition of $G = I_0(\M)$ adapted to $v_0$, and choose subgroups $A_1$, $A'$ of $A$ as above. Then we have
\begin{align}
& \left| \int_{A'} \int_N \langle f(n a' x_0), d_{na'} v_0 \rangle \langle \phi(n a' x_0), d_{na'} v_0 \rangle \,dn\,da' \right| \label{eq:basic_codim1_est} \\
\leq & C \Vert f(na'x_0)\Vert_{L^1(dnda')}^{1-\frac{1}{m}} \Vert f(n a x_0) \Vert_{L^1(dnda)}^{\frac{1}{m}}  \Vert \phi(a'n x_0)\Vert_{W^{1,m}(dnda')} \notag
\end{align}
where
$$
\Vert f(na'x_0)\Vert_{L^1(dnda')} := \int_{A'} \int_N |f(na'x_0)| \,dn\,da',
$$
$$
\Vert f(n a x_0) \Vert_{L^1(dnda)} := \int_A \int_N |f(n a x_0)| \,dn\,da,
$$
and
$$
 \Vert \phi(a' n x_0)\Vert_{W^{1,m}(dnda')} := \left(\int_{A'} \int_N |\phi(a' n x_0)|^m+|\nabla \phi(a' n x_0)|^m \,dn\,da' \right)^{\frac{1}{m}}.
$$
The constant $C$ is independent of the choice of $v_0$.
\end{prop}

Assume this proposition for the moment. Then for each $a_1 \in A_1$, we apply this inequality to the vector fields $f_{a_1}$ and $\phi_{a_1}$, where $f_{a_1}(x) := d_{a_1^{-1}} f(a_1 x)$, and $\phi_{a_1}(x) := d_{a_1^{-1}} \phi(a_1 x)$ for all $x \in \M$. We then get, on the left-hand side,
$$
\left|\int_{A'} \int_N \langle f(a_1 n a' x_0),d_{a_1 n a'} v_0 \rangle \langle \phi(a_1 n a' x_0),d_{a_1 n a'} v_0 \rangle \,dn\,da' \right|;
$$
This is because
$$\langle f_{a_1}(n a' x_0),d_{na'} v_0 \rangle = \langle d_{{a_1}^{-1}} f(a_1 n a' x_0), d_{na'} v_0 \rangle = \langle f(a_1 n a' x_0), d_{a_1 n a'} v_0 \rangle,$$
and similarly for $\phi_{a_1}$. Making the change of variable $n \mapsto a_1^{-1} n a_1$, the left hand side of \eqref{eq:basic_codim1_est} is thus
\begin{equation}
 e^{-2\rho(\log a_1)} \left|\int_{A'} \int_N \langle f(n a_1 a' x_0),d_{n a_1 a'} v_0 \rangle \langle \phi(n a_1 a' x_0),d_{n a_1 a'} v_0 \rangle \,dn\,da' \right|. \label{eq:lhsaverage}
\end{equation}

We now analyze the terms on the right hand side of \eqref{eq:basic_codim1_est} one by one. We have
\begin{align*}
\Vert f_{a_1}(na'x_0)\Vert_{L^1(dnda')}
&= \int_{A'} \int_N |f_{a_1}(na'x_0)| \,dn\,da' \\
&= \int_{A'} \int_N |f(a_1 n a' x_0)| \,dn\,da',
\end{align*}
since $$|f_{a_1}(n a' x_0)| = |d_{{a_1}^{-1}} f({a_1}n a' x_0)| = |f({a_1}n a' x_0)|
$$
(again the last equality follows from left-invariance of the Riemannian metric). Making again the change of variables $n \mapsto a_1^{-1} n a_1$, it follows that
\begin{equation} \label{eq:rhsaverage1}
\Vert f_{a_1}(na'x_0)\Vert_{L^1(dnda')} =  e^{-2\rho(\log a_1)} \int_{A'} \int_N |f(n a_1 a' x_0)| \,dn\,da'.
\end{equation}
Also,
\begin{equation*}
\begin{split}
\Vert f_{a_1}(n a x_0) \Vert_{L^1(dnda)}
&= \int_A \int_N |f(a_1 n a x_0)| \,dn\,da \\
&= e^{-2\rho(\log a_1)} \int_A \int_N |f(n a_1 a x_0)| \,dn\,da
\end{split}
\end{equation*}
(in the last line we made again the change of variables $n \mapsto a_1^{-1} n a_1$). Since $da$ is a Haar measure on $A$, this shows
\begin{equation}
\Vert f_{a_1}(n a x_0) \Vert_{L^1(dnda)} =  e^{-2\rho(\log a_1)} \int_A \int_N |f(nax_0)| \,dn\,da   \label{eq:rhsaverage2}
\end{equation}
Finally, we claim
\begin{multline}
\Vert \phi_{a_1}(a' n x_0)\Vert_{W^{1,m}(dnda')}  \\
= \left( \int_{A'} \int_N |\phi(a_1 a'nx_0)|^m + |(\nabla \phi) (a_1 a'nx_0)|^m \,dn\,da' \right)^{\frac{1}{m}}. \label{eq:rhsaverage3}
\end{multline}
Indeed, $|\phi_{a_1}(a'nx_0)| = |\phi(a_1 a' nx_0)|$, and we will show that $$|(\nabla \phi_{a_1}) (a'nx_0)| = |(\nabla \phi) (a_1 a' nx_0)|.$$ This is because if $\{X_{\ell}\}_{\ell=1}^m$ is an orthonormal frame at $a'nx_0$, then $\{d_{a_1} X_{\ell}\}_{\ell=1}^m$ is an orthonormal frame at $a_1 a' nx_0$, from which it follows that
$$
|(\nabla \phi_{a_1}) (a'nx_0)|^2
= \sum_{\ell=1}^m |(\nabla_{X_{\ell}} \phi_{a_1}) (a'nx_0)|^2
= \sum_{\ell=1}^m |(\nabla_{d_{a_1} X_{\ell}} \phi) (a_1 a'nx_0)|^2
$$
(We use also the left-invariance of the Levi-Civita connection $\nabla$ on $\M$.) Hence $|(\nabla \phi_{a_1}) (a'nx_0)| = |(\nabla \phi) (a_1a'nx_0)|$, and the identity \eqref{eq:rhsaverage3} follows.

Putting \eqref{eq:lhsaverage}, \eqref{eq:rhsaverage1}, \eqref{eq:rhsaverage2} and \eqref{eq:rhsaverage3} back into \eqref{eq:basic_codim1_est}, and multiplying both sides of \eqref{eq:basic_codim1_est} by $e^{2\rho(\log a_1)}$, we get
\begin{multline*}
 \left| \int_{A'} \int_N \langle f(na_1 a' x_0),d_{na_1 a'}v_0 \rangle \langle \phi(na_1 a' x_0), d_{na_1 a'}v_0 \rangle \,dn\,da' \right| \\
\leq  \, C \left( \int_{A'} \int_N |f(n a_1 a' x_0)| \,dn\,da' \right)^{1-\frac{1}{m}}
\left( \int_A \int_N |f(nax_0)| \,dn\,da \right)^{\frac{1}{m}} \\
 \qquad \cdot \left( \int_{A'} \int_N |\phi(a_1 a'nx_0)|^m + |(\nabla \phi) (a_1 a'nx_0)|^m \,dn\,da' \right)^{\frac{1}{m}}
\end{multline*}

We now integrate both sides of this equation over $A_1$ with respect to the Haar measure $da_1$. Then applying H\"older's inequality in the integral over $A_1$ on the right hand side, we obtain
\begin{multline*}
\left| \int_A \int_N \langle f(nax_0),d_{na} v_0 \rangle  \langle \phi(nax_0),d_{na}v_0 \rangle \,dn\,da\right|\\
\leq  \, C \left( \int_A \int_N |f(nax_0)| \,dn\,da \right) \hspace{2cm}\\
\cdot \left(\int_A \int_N |\phi(anx_0)|^m + |(\nabla \phi)(anx_0)|^m \,dn\,da \right)^{\frac{1}{m}}.
\end{multline*}
By applying this to $f_k$ and $\phi_k$, where
$f_k(x) := d_{k^{-1}} f(kx)$ and $\phi_k(x):= d_{k^{-1}} \phi(kx)$, we then have
\begin{multline*}
\left| \int_A \int_N \langle f(kna x_0),d_{kna}v_0 \rangle  \langle \phi(knax_0),d_{kna}v_0 \rangle  \,dn\,da\right|\\
\leq  \, C \left( \int_A \int_N |f(knax_0)| \,dn\,da \right)  \hspace{2cm}\\
\cdot \left(\int_A \int_N |\phi(kanx_0)|^m+|(\nabla \phi) (kanx_0)|^m \,dn\,da \right)^{\frac{1}{m}}.
\end{multline*}
In the last line, the integral involving $\phi$ is bounded by $\Vert \phi\Vert_{L^m(dV)} + \Vert \nabla \phi\Vert_{L^m(dV)}$. This is because $k$ acts by isometry on $\M$, which allows us to drop the $k$ in the integral, and then apply \eqref{eq:intAN=intM}. It follows that
\begin{multline*}
 \left| \int_A \int_N \langle f(kna x_0),d_{kna}v_0 \rangle  \langle \phi(knax_0),d_{kna}v_0 \rangle  \,dn\,da\right|\\
\leq \, C \left( \int_A \int_N |f(knax_0)| \,dn\,da \right)
\left(\Vert \phi\Vert_{L^m(dV)} +  \Vert \nabla \phi\Vert_{L^m(dV)} \right).
\end{multline*}
It now remains to integrate both sides over $K$. It follows that \eqref{eq:intKAN}, and hence the left hand side of \eqref{eq:intG_est}, are bounded by
$$
C \left( \int_K \int_A \int_N |f(knax_0)| \,dn\,da\,dk \right) \left(\Vert \phi\Vert_{L^m(dV)} +  \Vert \nabla \phi\Vert_{L^m(dV)} \right).
$$
But
\begin{align*}
\int_K \int_A \int_N |f(knax_0)| \,dn\,da\,dk
= \int_G |f(gx_0)| dg = \Vert f\Vert_{L^1(dV)},
\end{align*}
the first equality following from \eqref{eq:d(kna)}, the second from \eqref{eq:dgK=dV}.
Altogether, this shows that
\begin{multline*}
\left| \int_G \langle f(gx_0), d_g v_0 \rangle \langle \phi(gx_0), d_g v_0 \rangle dg \right|\\
\leq C \Vert f\Vert_{L^1(dV)} \left(\Vert \phi\Vert_{L^m(dV)} +  \Vert \nabla \phi\Vert_{L^m(dV)} \right).
\end{multline*}
This is almost the estimate \eqref{eq:intG_est} we wanted to prove, except that on the right hand side we have an additional $\Vert \phi\Vert_{L^m(dV)}$. Nevertheless, one can complete the proof of \eqref{eq:intG_est}, once one invokes the following Sobolev lemma, with $p = m$:

\begin{lem} \label{lem:Hardy}
For any compactly supported smooth vector field $\phi$ on $\M$, and any $1 \leq p < \infty$, we have
$$
\Vert \phi\Vert_{L^p(dV)} \leq C_p \Vert \nabla \phi\Vert_{L^p(dV)}.
$$
\end{lem}

\begin{proof}
Indeed, recall the following elementary inequality on $\mathbb{R}$: If $h$ is a smooth function with compact support on $\mathbb{R}$, then for any $\lambda > 0$, $1 \leq p < \infty$, we have
\begin{equation} \label{eq:HardyR}
\int_{\mathbb{R}} |h(\tau)|^p e^{-\lambda \tau} d\tau \leq \left(\frac{p}{\lambda}\right)^p \int_{\mathbb{R}} |h'(\tau)|^p e^{-\lambda \tau} d\tau.
\end{equation}
This inequality is just Hardy's inequality in disguise; since its proof is very simple, we reproduce it below. First, note that
$$
h(\tau) = \int_0^{\infty} h'(\tau-t) dt,
$$
from which it follows that
$$
|h(\tau)| e^{-\frac{\lambda}{p} \tau} = \left| \int_0^{\infty} h'(\tau-t) e^{-\frac{\lambda}{p}(\tau-t)} \, e^{-\frac{\lambda}{p} t} dt \right|.
$$
Taking $L^p(d\tau)$ norm of both sides, using Minkowski's inequality, and then raising to the $p$-th power, we obtain \eqref{eq:HardyR}.

The way we will apply \eqref{eq:HardyR} is as follows. Let $G = I_0(\M)$. Pick an Iwasawa decomposition $G = KAN$ of $G$. Write $da$ and $dn$ for the normalized Haar measures on $A$ and $N$ respectively, as in Section~\ref{sect:int}. Also fix $H \in \a$, the Lie algebra of $A$, such that
$$
\rho(H) > 0,
$$
where $\rho$ is half the sum of the positive roots (such $H$ clearly exists: just pick a non-zero vector in the interior of the positive Weyl chamber). We normalize $H$ so that it has unit length with respect to the Killing metric $g_0$. Now let $\a''$ be the orthogonal complement of $H$ in $\a$ with respect to $g_0$. Write $A''$ for the connected Lie group whose Lie algebra is $\a''$, and $da''$ for the normalized Haar measure on $A''$, so that
$$
\int_A F_1(a) da = \int_{\mathbb{R}} \int_{A''} F_1(a'' e^{\tau H}) \, da'' \, d\tau
$$
for any compactly supported continuous function $F_1$ on $A$.

Suppose $\Phi$ is a compactly supported smooth function on $\M$. For $n \in N$, $a'' \in A''$, we apply \eqref{eq:HardyR} to $h(\tau) := \Phi(n a'' e^{\tau H}x_0)$ and $\lambda:= 2\rho({H})$. Then again identifying $H \in \a$ with an $S$-invariant vector field on $\M$, where $S:=AN=NA$, we obtain
$$
\int_{\mathbb{R}} |\Phi(n a'' e^{\tau {H}} x_0)|^p e^{-2\rho(\tau {H})} d\tau \leq C_p^p \int_{\mathbb{R}} \left| (H \Phi)(n a'' e^{\tau {H}} x_0) \right|^p e^{-2 \rho(\tau {H})} d\tau
$$
for all $1 \leq p < \infty$, where $C_p = \frac{p}{2\rho({H})}$ (which is finite since we chose $H$ such that $\rho(H) > 0$). We now multiply both sides by $e^{-2\rho(\log a'')}$, and integrate over $a'' \in A''$ and $n \in N$ with respect to $da''$ and $dn$. This shows
$$
\int_N \int_A |\Phi(n a x_0)|^p e^{-2\rho(\log a)} da dn
\leq C_p^p \int_N \int_A \left| ({H} \Phi)(n a x_0) \right|^p e^{-2\rho(\log a)} da dn
$$
for all $1 \leq p < \infty$.
In view of \eqref{eq:intNA=intM}, this yields
\begin{equation} \label{eq:Hardy1}
\Vert \Phi\Vert_{L^p(dV)} \leq C_p \Vert {H} \Phi\Vert_{L^p(dV)}
\end{equation}
for all $1 \leq p < \infty$.

Now recall we had a good basis $\{X_{\ell}\}_{1 \leq \ell \leq m}$ of $\s$ from Section~\ref{sect:Kill_metric}, which we think of as a basis of $S$-invariant vector fields on $\M$. The vector fields $X_{\ell}$ satisfy
$\nabla_H X_{\ell} = 0$ by (\ref{eq:parallel}).
We apply \eqref{eq:Hardy1} to the function $\Phi:= \langle \phi, X_{\ell} \rangle$ on $\M$ for $1 \leq {\ell} \leq m$. Then we obtain
$$
{H} \Phi = \langle \nabla_{H} \phi, X_{\ell} \rangle +  \langle  \phi, \nabla_{H} X_{\ell} \rangle =  \langle \nabla_{H} \phi, X_{\ell} \rangle,
$$
so
$$
|{H} \Phi| \leq |\nabla \phi|.
$$
In view of \eqref{eq:Hardy1}, this implies
$$
\Vert \phi\Vert_{L^p(dV)} \leq C_p \Vert \nabla \phi\Vert_{L^p(dV)}
$$
for all $1 \leq p < \infty$, and this concludes the proof of Lemma~\ref{lem:Hardy}.
\end{proof}

It remains now to prove Proposition~\ref{prop:basicest}. Suppose $v_0 \in \mathbb{S}^{m-1}$.  Let $G =KAN$ be the Iwasawa decomposition of $G$ adapted to $v_0$, that we chose in the proof of Proposition~\ref{prop:G_int}, and let's identify $\M$ with the Lie group $S = NA$. So from now on, we think of $f$ and $\phi$ as vector fields on $S$ instead of vector fields on $\M$. Let $\{H_1, \dots, H_r, Y_1, \dots, Y_{m-r}\}$ be the good basis of the Lie algebra $\s$ of $S$, that we chose immediately after equation \eqref{eq:intKAN}. Again we label these as $\{X_1, \dots, X_m\}$, as in \eqref{eq:X_label}. We express $f$ and $\phi$ in terms of this basis: say $$f = \sum_{i=1}^m f^i X_i, \quad \phi = \sum_{i=1}^m \phi^i X_i,$$ where $f^1, \dots, f^m \in C^{\infty}(S)$ and $\phi^1, \dots, \phi^m \in C^{\infty}_c(S)$. Then since $X_1 = {H_1}$ and $v_0 = \left. \frac{d}{dt} \right|_{t=0} e^{t{H_1}} x_0$, we have $X_1$ at $nax_0$ equal to $d_{na}v_0$; since $X_1$ has unit length, this shows $$f^1(na) = \langle f(nax_0), d_{na} v_0 \rangle, \quad \phi^1(na) = \langle \phi(nax_0), d_{na} v_0 \rangle.$$ Hence the desired inequality \eqref{eq:basic_codim1_est} in Proposition~\ref{prop:basicest} translates into
\begin{align}
& \left| \int_{A'} \int_N f^1(na') \phi^1(na')  \,dn\,da'  \right| \notag \\
\leq & \, C \Vert f(na')\Vert_{L^1(dnda')}^{1-\frac{1}{m}} \Vert f(na)\Vert_{L^1(dnda)}^{\frac{1}{m}} \Vert  (|\phi|+|\nabla \phi|)(a'n)\Vert_{L^m(dnda')}. \label{eq:basic_codim1_est_2}
\end{align}
This is to be established under the condition $\div f = 0$, which translates into
\begin{equation} \label{eq:divf_explicit2}
\sum_{\ell=1}^m X_{\ell} f^{\ell} = 2 \sum_{i=1}^r \rho({H_i}) f^{i} \quad \text{on $S$}
\end{equation}
by \eqref{eq:divf_explicit}.

To see that \eqref{eq:basic_codim1_est_2} is true under hypothesis \eqref{eq:divf_explicit2}, we apply Lemma~\ref{lem:decompS'unif}  with $p = m$ to the function $\Phi = \phi^1$ on $S' = A'N$.
Then for any $\lambda > 0$, there exists a decomposition
$$
\phi^1 = \Phi_1 + \Phi_2, \quad \text{on $S'$},
$$
such that $\Phi_1, \Phi_2 \in C^{\infty}_c(S')$, and \eqref{eq:decomp_conclu} holds with $p=m$. Now we split the integral on the left hand side of \eqref{eq:basic_codim1_est_2} into two pieces:
\begin{align*}
&\int_{A'} \int_N f^1(na') \phi^1(na') \,dn\,da'  \\
=& \int_{A'} \int_N f^1(na') \Phi_1(na') \,dn\,da' + \int_{A'} \int_N f^1(na') \Phi_2(na') \,dn\,da' \\
=& I + II.
\end{align*}
The first integral $I$ is easy to estimate:
\begin{align}
|I| & \leq \Vert f(na')\Vert_{L^1(dnda')} \Vert \Phi_1\Vert_{L^{\infty}(S')}  \notag \\
& \leq C \lambda^{\frac{1}{m}} \Vert f(na')\Vert_{L^1(dnda')} \Vert \phi^1 \Vert_{W^{1,m}(S')}. \label{eq:estI}
\end{align}
To estimate the second integral $II$, fix some $\chi \in C^{\infty}_c(A_1)$ with $\chi = 1$ at the identity. We rewrite $II$ as
\begin{align*}
II &= -\int_{A'} \int_N \int_0^{\infty} \frac{d}{dt} [f^1(n a' e^{t {H_1}}) \chi(e^{t {H_1}})] \Phi_2(na') dt \,dn\,da' \\
&= -\int_{A'} \int_N \int_0^{\infty} [({H_1} f^1)(n a' e^{t {H_1}}) \chi(e^{t {H_1}}) \Phi_2(na') \\
&\qquad \qquad + f^1(n a' e^{t {H_1}}) ({H_1} \chi)(e^{t {H_1}}) \Phi_2(na')] dt \,dn\,da' \\
&= \int_{A'} \int_N \int_0^{\infty} \sum_{\ell=2}^{m} (X_{\ell} f^{\ell})(n a' e^{t {H_1}})  \chi(e^{t {H_1}}) \Phi_2(na') dt \,dn\,da'\\
& \qquad - 2 \sum_{i=1}^r \rho({H_i}) \int_{A'} \int_N \int_0^{\infty} f^i(n a' e^{t {H_1}})  \chi(e^{t {H_1}}) \Phi_2(na') dt \,dn\,da' \\
& \qquad - \int_{A'} \int_N \int_0^{\infty} f^1(n a' e^{t {H_1}}) ({H_1} \chi)(e^{t {H_1}}) \Phi_2(na') dt \,dn\,da'
\end{align*}
where we have used our hypothesis \eqref{eq:divf_explicit2} in the last equality. The last two integrals are obviously bounded by
\begin{equation} \label{eq:IIest3}
C \Vert f(na)\Vert_{L^1(dnda)} \Vert \Phi_2\Vert_{L^{\infty}(S')}.
\end{equation}
To estimate the first, note that one can easily integrate by parts for those terms corresponding to $2 \leq \ell \leq r$: we have, for these $\ell$'s, that $X_{\ell} f^{\ell} = H_{\ell} f^{\ell}$, so using Lemma~\ref{lem:intYi} to integrate by parts in $A'$, we get
\begin{multline}
\int_{A'} \int_N \int_0^{\infty} (X_{\ell} f^{\ell})(n a' e^{t {H_1}}) \chi(e^{t {H_1}}) \Phi_2(na') dt \,dn\,da'  \\
= - \int_{A'} \int_N \int_0^{\infty}  f^{\ell}(n a' e^{t {H_1}})  \chi(e^{t {H_1}}) (X_{\ell} \Phi_2)(na') dt \,dn\,da' \label{eq:IIest1}
\end{multline}
for $2 \leq \ell \leq r$.
Now for $r+1 \leq \ell \leq m$, we have $X_{\ell} f^{\ell} = Y_j f^{r+j}$ where $j = \ell-r$. For any $1 \leq j \leq m-r$ and any smooth function $F$ on $S$, if $a_1 \in A_1$, $s' \in S'$, we have
$$
(Y_j F)(s'a_1)
= \left. \frac{d}{d\tau} \right|_{\tau=0} F(s' a_1 e^{\tau Y_j})
= \left. \frac{d}{d\tau} \right|_{\tau=0} (R_{a_1} F)(s' a_1 e^{\tau Y_j} a_1^{-1})
$$
where $R_{a_1} F(x):= F(xa_1)$ is the right translation by $a_1$. But then this is equal to
\begin{align*}
[(\Ad a_1)(Y_j)](R_{a_1} F)(s')
&= ( e^{\ad (\log a_1)} Y_j)(R_{a_1} F)(s') \\
&= e^{\alpha_j(\log a_1)} [Y_j(R_{a_1} F)](s').
\end{align*}
Hence for each $r+1 \leq \ell \leq m$, and each $t \in \mathbb{R}$, we have
$$
\int_N (Y_{\ell-r} f^{\ell})(n a' e^{t {H_1}}) \Phi_2(na') dn
= e^{t\alpha_{\ell -r}({H_1})} \int_N [Y_{\ell-r} (R_{e^{t {H_1}}} f^{\ell})](na') \Phi_2(na') dn,
$$
which by Lemma~\ref{lem:intYi} is equal to
$$
-e^{t\alpha_{\ell -r}({H_1})} \int_N (R_{e^{t {H_1}}} f^{\ell})(na') (Y_{\ell-r} \Phi_2)(na') dn.
$$
This shows that
\begin{multline}
\int_{A'} \int_N \int_0^{\infty} (X_{\ell} f^{\ell})(n a' e^{t {H_1}})  \chi(e^{t {H_1}}) \Phi_2(na') dt \,dn\,da'  \\
= - \int_0^{\infty} \int_{A'} \int_N  f^{\ell}(n a' e^{t {H_1}}) e^{t\alpha_{\ell -r}({H_1})} \chi(e^{t {H_1}}) (X_{\ell} \Phi_2)(na') \,dn\,da' dt
\label{eq:IIest2}
\end{multline}
for $r+1 \leq \ell \leq m$.
Combining \eqref{eq:IIest3}, \eqref{eq:IIest1} and \eqref{eq:IIest2}, we see that
\begin{equation}
\begin{split}
|II| &\leq C \Vert f(na)\Vert_{L^1(dnda)} \left( \Vert \nabla' \Phi_2\Vert_{L^{\infty}(S')} + \Vert \Phi_2\Vert_{L^{\infty}(S')} \right) \\
& \leq C \lambda^{\frac{1}{m}-1} \Vert f(na)\Vert_{L^1(dnda)} \Vert  \phi^1\Vert_{W^{1,m}(S')}. \label{eq:estII}
\end{split}
\end{equation}
(Here we used that $e^{t \alpha_{\ell}(H_1)} \chi(e^{tH_1})$ is a bounded function of $t$.) Adding \eqref{eq:estI} and \eqref{eq:estII}, and optimizing $\lambda$, we get
\begin{multline*}
\left| \int_{A'} \int_N f^1(na') \phi^1(na') \,dn\,da' \right|  \\
\leq \, C \Vert f(na')\Vert_{L^1(dnda')}^{1-\frac{1}{m}} \Vert f(na)\Vert_{L^1(dnda)}^{\frac{1}{m}} \Vert  \phi^1\Vert_{W^{1,m}(S')}.
\end{multline*}
The constant $C$ here comes primarily from Lemma~\ref{lem:decompS'unif}. It is an absolute constant independent of the Iwasawa decomposition $KAN$ of $G$ we chose, and independent of the good basis of $\s$ as well. Hence it is independent of $v_0 \in \mathbb{S}^{m-1}$. It remains to observe that
$$
 \Vert  \phi^1\Vert_{W^{1,m}(S')} \leq C \Vert  (|\phi|+|\nabla \phi|)(a'n)\Vert_{L^m(dnda')}:
$$
indeed
$$
|(\nabla' \phi^1)(a'n)| \leq C (|\phi|+|\nabla \phi|)(a'n),
$$
which holds since
$$
\phi^1(a'n) = \langle \phi(a'n x_0), d_{a'n} v_0 \rangle
$$
and we identified $\phi(a'n)$ with $\phi(a'nx_0)$.
This proves \eqref{eq:basic_codim1_est_2} with a constant $C$ independent of $v_0$, and hence concludes our proof of Proposition~\ref{prop:basicest}. Proposition~\ref{prop:G_int}, and thus Theorem~\ref{thm:main}, are now established in full.

\section{Proof of Theorem~\ref{thm:Laplace_application}} \label{sect:applications}

To prove Theorem~\ref{thm:Laplace_application}, we need some facts about the Riesz transforms on $(\M,g)$. Recall that the $L^2$ spectrum of $\Delta$ is the half-line $(-\infty,-|\rho|^2]$, and in particular stays away from zero. Thus one can define negative powers of $-\Delta$, and they are bounded self-adjoint operators on the space of $L^2$ functions on $\M$. It is known that $\nabla^2 (-\Delta)^{-1}$ extends to a bounded linear operator on $L^p(dV)$ for all $1 < p < \infty$; here $\nabla^2$ is any second order derivative on $\M$, i.e. any $X_i X_j$ where $X_i, X_j$ are smooth vector fields that have lengths at most 1 on $\M$ (see e.g. Lohou\'e~\cite[Theorem II]{MR786621}). Combined with our Lemma~\ref{lem:Hardy}, we also see that $\nabla (-\Delta)^{-1}$ and $(-\Delta)^{-1}$ are bounded on $L^p(dV)$ for all $1 < p < \infty$, i.e.
\begin{equation} \label{eq:fracint1}
\|X_i (-\Delta)^{-1} h\|_{L^p(dV)} \lesssim \|h\|_{L^p(dV)},
\end{equation}
and
\begin{equation} \label{eq:fracint2}
\|(-\Delta)^{-1} h\|_{L^p(dV)} \lesssim \|h\|_{L^p(dV)}
\end{equation}
for $1 < p < \infty$ and $h \in L^p(dV)$.

Suppose $f$ is a smooth vector field on $\M$ with $\div f = 0$, and $\varphi$ is a compactly supported smooth vector field on $\M$ with $\|\varphi\|_{L^{\infty}(dV)} + \|\nabla \varphi\|_{L^{\infty}(dV)} \leq 1$. Let $u \in L^2(dV)$ be the solution of the equation
$$
\Delta u = \langle f, \varphi \rangle.
$$
To study $\nabla u$, we pick an orthonormal basis of global vector fields $X_1, \dots, X_m$ on $(\M,g)$ as follows. Let $G = I_0(\M)$, and pick an Iwasawa decomposition $G = ANK$. Then $\M = G/K$ is diffeomorphic to the Lie group $S:=AN$, and the Riemannian metric $g$ on $\M$ induces a left-invariant metric on $S$, so we can pick an orthonormal basis of global vector fields $X_1, \dots, X_m$ on $\M$, that also forms a basis of left-invariant vector fields on $S$. It follows that for any $1 \leq i, j \leq m$, the commutator of the vector fields $X_i$ and $X_j$ is a constant linear combination of $X_1, \dots, X_m$; in other words, there exists constants $c_{ij}^k$, for $1 \leq i,j,k \leq m$, such that $[X_i,X_j] = \sum_{k=1}^m c_{ij}^k X_k$ for all $1 \leq i, j \leq m$. In particular, by a similar calculation to the one in Section~\ref{sect:div}, $\div X_i$ is a constant, say $c_i$, for any $1 \leq i \leq m$. Then writing $$X_i^* := -X_i+c_i,$$ we have $\int_{\M} (X_i u) \, h \, dV = \int_{\M} u \, (X_i^* h) \, dV$ for all test functions $u$ and $h$, and any $1 \leq i \leq m$. Now
$$
-\Delta (X_i u) = X_i (-\Delta u) + [X_i, -\Delta] u,
$$
and $[X_i, -\Delta]$ is a second order operator of the form $\sum_{j,k=1}^m a_{jk} X_j X_k + \sum_{j=1}^m b_j X_j$ for some constants $a_{jk}$ and $b_j$. Thus writing $F = -\Delta u = - \langle f, \varphi \rangle$, we get
$$
X_i u = (-\Delta)^{-1} X_i F + (-\Delta)^{-1} [X_i, -\Delta] (-\Delta)^{-1} F,
$$
and hence
\begin{align} \label{eq:Xiu_dual}
 \int_{\M} (X_i u) \, h \, dV & = \int_{\M} F \, X_i^* (-\Delta)^{-1} h \, dV  + \int_{\M} F \, (-\Delta)^{-1} [X_i, -\Delta]^* (-\Delta)^{-1} h \, dV 
\end{align}
for any test function $h \in C^{\infty}_c(\M)$, where $[X_i, -\Delta]^* = \sum_{j,k=1}^m a_{jk} X_k^* X_j^* + \sum_{j=1}^m b_j X_j^*$. The first term on the right hand side is just
$$
\int_{\M} \langle f, [X_i^* (-\Delta)^{-1} h] \varphi \rangle dV,
$$
and can be estimated by Theorem~\ref{thm:main} by taking $\phi = [X_i^* (-\Delta)^{-1} h] \varphi$; since
$$
\nabla \phi = [\nabla X_i^* (-\Delta)^{-1} h] \varphi + [X_i^* (-\Delta)^{-1} h] (\nabla \varphi),
$$
and since $\|\varphi\|_{L^{\infty}(dV)} + \|\nabla \varphi\|_{L^{\infty}(dV)} \leq 1$, by the result about the Riesz transforms cited above, as well as (\ref{eq:fracint1}) and (\ref{eq:fracint2}), we see that
$$
\|\nabla \phi\|_{L^m(dV)} \leq C \|h\|_{L^m(dV)}.
$$
This shows that the first term on the right hand side of (\ref{eq:Xiu_dual}) is bounded by $C\|f\|_{L^1(dV)} \|h\|_{L^m(dV)}$.
Similarly, the second term on the right hand side of (\ref{eq:Xiu_dual}) is bounded by $C\|f\|_{L^1(dV)} \|h\|_{L^m(dV)}$: indeed there we take $\phi$ to be $[(-\Delta)^{-1} [X_i, -\Delta]^* (-\Delta)^{-1} h] \varphi$, and then
$$
\nabla \phi = [\nabla (-\Delta)^{-1} [X_i, -\Delta]^* (-\Delta)^{-1} h] \varphi + [(-\Delta)^{-1} [X_i, -\Delta]^* (-\Delta)^{-1} h] \nabla \varphi,
$$
so
$$
\|\nabla \phi\|_{L^m(dV)} \leq C \|h\|_{L^m(dV)}
$$
by (\ref{eq:fracint1}), (\ref{eq:fracint2}) and the cited result about the Riesz transforms.
Altogether, this shows
$$
\left| \int_{\M} (X_i u) \, h\, dV \right| \leq C \|f\|_{L^1(dV)} \|h\|_{L^m(dV)}
$$
for all test functions $h \in C^{\infty}_c(\M)$, and hence
$$
\|\nabla u\|_{L^{\frac{m}{m-1}}(dV)} \leq C \|f\|_{L^1(dV)},
$$
as desired.

\end{document}